\documentclass[11pt]{article}
\usepackage{graphicx}
\usepackage{amsmath} 
\usepackage{wrapfig} 
\usepackage{amssymb} 
\usepackage[margin=1in]{geometry}
\usepackage[colorlinks]{hyperref}
\usepackage{enumitem}
\usepackage{titlesec}
\usepackage{cite}
\usepackage{setspace}
\usepackage{fancyhdr}
\usepackage{booktabs}
\usepackage{lipsum}   
\usepackage{lastpage}
\usepackage{etoolbox} 
\usepackage{amsthm}
\usepackage{tikz}
\usepackage{caption}
\usepackage{subcaption}
\usepackage{cleveref}
\usepackage{tikz-cd}
\usepackage{array}
\usepackage{makecell}
\usepackage{pifont}

\numberwithin{equation}{section}

\usepackage[labelfont=bf, textfont=bf]{caption} 

\newtheorem{theorem}{Theorem}[section]
\newtheorem{lemma}[theorem]{Lemma}
\newtheorem*{maintheorem}{Main Theorem}
\newtheorem{corollary}[theorem]{Corollary}
\newtheorem{proposition}[theorem]{Proposition}
\newtheorem{conjecture}[theorem]{Conjecture}
\theoremstyle{definition}

\newtheorem{question}[theorem]{Question}

\newtheorem{remark}{Remark}

\let\oldcite\cite
\renewcommand{\cite}[1]{\textup{\!\!\oldcite{#1}}}

\begin{document}
\title{Topological Insights and Combinatorial Frameworks for Twist Knot Complements}

\author{Biplab Basak$^1$, Mangaldeep Saha, and Ayushi Trivedi}
	
\date{\today}
	
\maketitle
	
\vspace{-10mm}
\begin{center}
		
\noindent {\small Department of Mathematics, Indian Institute of Technology Delhi, New Delhi, 110016, India.$^2$}

\footnotetext[1]{Corresponding author}
		
\footnotetext[2]{{\em E-mail addresses:}  \url{biplab@iitd.ac.in} (B. Basak), \url{maz238698@maths.iitd.ac.in} (M. Saha), \url{maz238453@maths.iitd.ac.in} (A. Trivedi).}
		
\medskip
	
\end{center}
	
\hrule
	
\begin{abstract}
The complement of a knot in $\mathbb{S}^3$ is a fundamental object in $3$-manifold topology and serves as a primary building block in the study of geometric structures on manifolds. For a large class of knots, these complements admit complete hyperbolic metrics, motivating the development of various frameworks for their classification and analysis. A twist knot is formed by applying a series of half-twists to a closed loop and linking the ends. In this article, we present explicit topological and combinatorial descriptions of all twist knot complements via the construction of suitable cell structures. Our method produces a closed $3$-dimensional pseudomanifold with a single vertex, arising from prescribed face pairings of a faceted $3$-ball. Removing the vertex yields the complement of the corresponding twist knot. The construction naturally induces ideal triangulations of these knot complements, several of which are shown to be facet-minimal. Beyond providing concrete models, our approach offers new perspectives on the topology of knot complements and leads to some open research problems.
\end{abstract}

\noindent {\small {\em MSC 2020\,:} Primary 57M50; Secondary 57K10,  57K30, 57K32, 57M05, 57Q15.

		
}
	
\medskip

\section{Introduction}
The mathematical study of knots is an intersection of geometric analysis, algebraic topology, and three-dimensional manifold theory. A knot is any smooth embedding of a circle $\mathbb{S}^{1}$ into the $3$-sphere $\mathbb{S}^{3}$. Two knots are equivalent if an ambient isotopy of $\mathbb{S}^3$ takes one to the other. A fundamental problem in knot theory is to classify knots by distinguishing inequivalent embeddings via suitable invariants.
The knot complement, denoted as $\mathbb{S}^3 \setminus K$, represents the manifold resulting from the removal of the knot $K$ from the $3$-sphere. According to the Gordon-Luecke theorem \cite{mca1989knots}, the complement serves as a complete invariant because two knots in $\mathbb{S}^3$ are equivalent if and only if their complements are homeomorphic.  This topological equivalence allows for the classification of knots using 3-manifold tools such as fundamental groups and hyperbolic structures, which have facilitated the exhaustive tabulation of prime knots up to 20 crossings as of June 2025 \cite{thistlethwaite25}.

A twist knot is a specific knot type formed by applying a series of half-twists to a closed loop and linking the ends. Since any number of twists can be introduced, this process generates an infinite family of knots. A twist knot with $n$ half-twists has crossing number $m = n + 2$. By convention, the knot is denoted as $m_1$ if $m$ is even or equal to 3; if $m$ is odd and greater than 4, it is denoted as $m_2$. A thorough and comprehensive study of their topological and combinatorial characterization, however, seems to be missing in literature. In this work, we aim to address this gap by developing a series of results that are relevant to both theoretical and computational aspects of mathematics.

A \emph{face-pairing} $\varphi$ on the faceted $3$-ball $B^{3}$ works by matching up the faces ($2$-cells) on its boundary $\partial(B^{3})$ in pairs so that each face is identified with another. This process creates a quotient space $M = B^{3}/\varphi$ that naturally inherits a cell complex structure. In the induced cell structure on $M$, the \emph{link} of a vertex (0-cell) is defined as the boundary of a sufficiently small closed neighborhood of that vertex; equivalently, a closed neighborhood of the vertex is homeomorphic to the cone on its link. In particular, the link of any vertex in the cell complex $M = B^{3}/\varphi$ is a closed, connected surface. While every closed $3$-manifold can be represented by a face-pairing $\varphi$ on a faceted $3$-ball $B^{3}$, the result is technically a $3$-dimensional pseudomanifold; it only becomes a closed connected manifold if the link around every vertex is a $2$-sphere. If the quotient complex $M$ is a closed 3-dimensional pseudomanifold with singular vertices, each having a link that is a closed 2-manifold distinct from $ \mathbb{S}^{2}$, then by removing these vertices we obtain a non-compact $3$-manifold.

Triangulations of $3$-manifolds provide a concrete combinatorial approach to understanding their topological and geometric structure. In particular, an \emph{ideal triangulation} of a non-compact $3$-manifold $M$ is obtained by gluing ideal tetrahedra so that the resulting space is homeomorphic to $M$. This idea was first introduced by Thurston \cite{thurston1997three}, who showed that the figure-eight knot complement can be realized by suitably gluing just two ideal tetrahedra. Later, Moise \cite{moise52} and Bing \cite{bing59} proved that the interior of any compact $3$-manifold with toroidal boundary admits an ideal triangulation.

Over time, many explicit constructions of ideal triangulations for knot complements have been developed. In \cite{dadd2016constructing}, Dadd and Duan constructed infinitely many ideal triangulations of the figure-eight knot complement, and subsequent works \cite{aribi23, ibarra25, jaco2020minimal} extended these constructions to twisted and double-twisted knot complements using Thurston’s algorithm. More recently, \cite{bst1} presented ideal triangulations of the complements of $(p,2)$ torus knots. In higher dimensions, Budney et al.\ \cite{budney12} studied triangulations of two-dimensional knots and showed that one of the Cappell-Shaneson knot complements admits a remarkably small topological triangulation consisting of only two $4$-simplices. Despite these advances, constructing minimal triangulations of $3$-manifolds remains a challenging problem.

In this article, we investigate the topological and combinatorial structure of twist knot complements through an explicit construction obtained by gluing the faces of a faceted $3$-ball. The approach presented here is completely elementary and provides a conceptually transparent alternative to existing constructions based on Thurston’s algorithm, such as those developed in \cite{aribi23}.  We consider a cellular decomposition of a faceted $3$-ball as illustrated in Figure \ref{m.cell-PTFP}. Through an appropriate identification, we obtain a pseudomanifold $M$ with a single vertex. Removing this vertex results in a non-compact 3-manifold, denoted by $M_{0}$, which corresponds topologically to a knot complement. More precisely, we construct a cellular structure for the complements of all twist knots. This construction provides a systematic and geometrically transparent method for producing triangulations of knot complements, thereby offering an alternative perspective to existing approaches derived from Thurston’s framework. Then we construct an explicit triangulation of $M$, which naturally induces a triangulation of the associated manifold $M_{0}$. Among of these triangulations, we have listed some triangulations that are shown to be facet-minimal. We believe that our approach has several notable strengths: it is straightforward to implement, offers clear geometric insight into the resulting constructions, and admits natural generalizations to broader classes of knots and links. In particular, it allows for explicit control over the face identifications and facilitates computational verification of geometricity using tools such as \texttt{Regina} \cite{burton2013regina}.

The paper is organized as follows. Section~\ref{preliminaries} recalls the necessary background material and fixes the notation used throughout the paper. In Section~\ref{construct_manifolds}, we introduce a method for constructing the $3$-sphere from faceted $3$-balls via suitable face-pairing identifications. Section~\ref{twist.knot.complement} develops the topological construction of twist knot complements, which forms the foundation of our approach. In Section \ref{fundamental_group_knot}, we determine a general form of the fundamental group for all twist knot complements. These results are brought together in Section~\ref{main_theorem_proof}, where we prove our first main theorem. In Section~\ref{ideal_triangulation}, we establish our second main result on the ideal triangulations of twist knot complements, and provide explicit triangulations for selected examples, together with their isomorphism signatures. The minimality of certain ideal triangulations is addressed in Section~\ref{minimality}. Finally, Section~\ref{conclusion} concludes the paper with a discussion of how these results relate to existing questions on face-pairing constructions and triangulations, and outlines possible directions for future research.

\section{Preliminaries}\label{preliminaries}
A topological space \(M\) equipped with cellular structure is called a \emph{closed pseudomanifold of dimension~$3$} if it is Hausdorff, second-countable, and compact, and if the boundary of a small closed neighbourhood of every vertex (0-cell) is a closed, connected surface. A simplicial version of a 3-dimensional closed pseudomanifold can be found in \cite{BS2020}.

\begin{lemma}[\cite{cannon2000introduction}]\label{pm-m}
Let $M$ be a closed pseudomanifold of dimension $3$ obtained from a face-pairing. Let~$\chi$ denote the Euler characteristic.  
Then $M$ is a manifold if and only if $\chi(M)=0$.
\end{lemma}

\noindent To show that a closed $3$-manifold is homeomorphic to the $3$-sphere $\mathbb{S}^{3}$, it is sufficient to prove that its fundamental group is trivial. Indeed, the Poincar\'e conjecture states that any closed and simply connected $3$-manifold must be homeomorphic to $\mathbb{S}^{3}$. Since this conjecture was proved by Perelman in 2003, establishing the triviality of the fundamental group allows us to conclude that the manifold is a $3$-sphere. The following proposition states this result.

\begin{proposition}[\cite{Poincare1, Poincare2}]\label{propo:sphere}
A simply-connected closed $3$-manifold is homeomorphic to the $3$-sphere.
\end{proposition}

\subsection{Knots and their properties}
Knots provide a powerful framework for studying three-dimensional spaces, since every closed $3$-manifold can be obtained by surgery along a knot or link. We begin by recalling some basic notions from knot theory. Following \cite{schultens2014introduction}, a \emph{knot} in $\mathbb{S}^3$ is defined as a smooth isotopy class of smooth embeddings of the circle $\mathbb{S}^1$ into $\mathbb{S}^3$. Two knots $K$ and $K'$ in $\mathbb{S}^3$ are said to be \emph{equivalent} if there exists a homeomorphism $h : \mathbb{S}^3 \to \mathbb{S}^3$ such that $h(K) = K'$.  We denote such an equivalence class by $K$. A knot is called \emph{prime} if it cannot be expressed as the connected sum of two nontrivial knots; that is, it is indecomposable with respect to the connected sum operation. A \emph{twist knot} is a specific knot type formed by applying a series of half-twists to a closed loop and linking the ends. If the number of crossings $m$ is even or equal to $3$, the corresponding twist knot is denoted by $m_1$; whereas if $m$ is odd and greater than $4$, it is denoted by $m_2$. All twist knots are prime. The necessary preliminary lemmas regarding knots are as follows:

\begin{lemma}\cite{mca1989knots}\label{equivalent_knot}
    Two knots are equivalent if and only if their complements are homeomorphic.
\end{lemma}

\begin{lemma}\cite{FeustelWhitten1978}\label{not composite}
If two knot complements have isomorphic fundamental groups, then both knots are either prime or composite.
\end{lemma}

\begin{lemma}\cite{mca1989knots}\label{prime_knot}
If two prime knot complements have isomorphic fundamental groups, then they are equivalent.
\end{lemma}

\subsubsection{Algorithm for Computing the Fundamental Group of a Knot Complement}\label{Algorithm}
The fundamental group of a knot complement, also called the \emph{knot group}, 
$\pi_1(\mathbb{S}^3 \setminus K)$, can be computed using the Wirtinger presentation 
(cf. Lemma \ref{ref: wirtinger presentation}). For that, one first decomposes the knot diagram into arcs, each running from one undercrossing to the next; the number of such arcs equals the number of crossings. These arcs give generators $x_1, x_2, \dots, x_k$ of the knot group. Further, each crossing contributes a relator, yielding the same number of relators as crossings. To determine the relators $r_i$, we first fix an orientation of the knot and represent each crossing using the corresponding arcs in the knot diagram. In this representation, the overcrossing arc is drawn vertically and the undercrossing arc horizontally. A directed arrow is then assigned by following the direction of the arc, conventionally pointing from right to left. The relator $r_i$ is obtained by reading off the corresponding loop relation at the crossing.

\begin{center}

\tikzset{every picture/.style={line width=0.75pt}} 

\begin{tikzpicture}[x=0.75pt,y=0.75pt,yscale=-1,xscale=1]

\draw    (124.09,698.8) -- (124.09,768.8) -- (124.09,843.15) ;
\draw [shift={(124.09,845.15)}, rotate = 270] [color={rgb, 255:red, 0; green, 0; blue, 0 }  ][line width=0.75]    (10.93,-3.29) .. controls (6.95,-1.4) and (3.31,-0.3) .. (0,0) .. controls (3.31,0.3) and (6.95,1.4) .. (10.93,3.29)   ;
\draw    (43,769.02) -- (72.3,769.02) -- (87.93,769.02) -- (115.28,769.02) ;
\draw    (132.86,769.39) -- (160.21,769.76) -- (197.28,769.76) ;
\draw [shift={(199.28,769.76)}, rotate = 180] [color={rgb, 255:red, 0; green, 0; blue, 0 }  ][line width=0.75]    (10.93,-3.29) .. controls (6.95,-1.4) and (3.31,-0.3) .. (0,0) .. controls (3.31,0.3) and (6.95,1.4) .. (10.93,3.29)   ;
\draw [color={rgb, 255:red, 74; green, 144; blue, 226 }  ,draw opacity=1 ]   (162.07,803.02) -- (162.16,725.19) ;
\draw [shift={(162.16,723.19)}, rotate = 90.07] [color={rgb, 255:red, 74; green, 144; blue, 226 }  ,draw opacity=1 ][line width=0.75]    (10.93,-3.29) .. controls (6.95,-1.4) and (3.31,-0.3) .. (0,0) .. controls (3.31,0.3) and (6.95,1.4) .. (10.93,3.29)   ;
\draw [color={rgb, 255:red, 74; green, 144; blue, 226 }  ,draw opacity=1 ]   (87.02,799.95) -- (87.07,724.34) ;
\draw [shift={(87.07,722.34)}, rotate = 90.03] [color={rgb, 255:red, 74; green, 144; blue, 226 }  ,draw opacity=1 ][line width=0.75]    (10.93,-3.29) .. controls (6.95,-1.4) and (3.31,-0.3) .. (0,0) .. controls (3.31,0.3) and (6.95,1.4) .. (10.93,3.29)   ;
\draw [color={rgb, 255:red, 74; green, 144; blue, 226 }  ,draw opacity=1 ]   (97,720.8) -- (152,721.76) ;
\draw [shift={(154,721.8)}, rotate = 181.01] [color={rgb, 255:red, 74; green, 144; blue, 226 }  ,draw opacity=1 ][line width=0.75]    (10.93,-3.29) .. controls (6.95,-1.4) and (3.31,-0.3) .. (0,0) .. controls (3.31,0.3) and (6.95,1.4) .. (10.93,3.29)   ;
\draw [color={rgb, 255:red, 74; green, 144; blue, 226 }  ,draw opacity=1 ]   (93.79,814.02) -- (153,814.77) ;
\draw [shift={(155,814.8)}, rotate = 180.73] [color={rgb, 255:red, 74; green, 144; blue, 226 }  ,draw opacity=1 ][line width=0.75]    (10.93,-3.29) .. controls (6.95,-1.4) and (3.31,-0.3) .. (0,0) .. controls (3.31,0.3) and (6.95,1.4) .. (10.93,3.29)   ;
\draw    (246,770.02) -- (275.3,770.02) -- (290.93,770.02) -- (318.28,770.02) ;
\draw    (335.86,770.39) -- (363.21,770.76) -- (400.28,770.76) ;
\draw [shift={(402.28,770.76)}, rotate = 180] [color={rgb, 255:red, 0; green, 0; blue, 0 }  ][line width=0.75]    (10.93,-3.29) .. controls (6.95,-1.4) and (3.31,-0.3) .. (0,0) .. controls (3.31,0.3) and (6.95,1.4) .. (10.93,3.29)   ;
\draw [color={rgb, 255:red, 74; green, 144; blue, 226 }  ,draw opacity=1 ]   (366.12,800.69) -- (366,727.6) ;
\draw [shift={(366,725.6)}, rotate = 89.91] [color={rgb, 255:red, 74; green, 144; blue, 226 }  ,draw opacity=1 ][line width=0.75]    (10.93,-3.29) .. controls (6.95,-1.4) and (3.31,-0.3) .. (0,0) .. controls (3.31,0.3) and (6.95,1.4) .. (10.93,3.29)   ;
\draw [color={rgb, 255:red, 74; green, 144; blue, 226 }  ,draw opacity=1 ]   (288.98,803.18) -- (289,729.6) ;
\draw [shift={(289,727.6)}, rotate = 90.02] [color={rgb, 255:red, 74; green, 144; blue, 226 }  ,draw opacity=1 ][line width=0.75]    (10.93,-3.29) .. controls (6.95,-1.4) and (3.31,-0.3) .. (0,0) .. controls (3.31,0.3) and (6.95,1.4) .. (10.93,3.29)   ;
\draw [color={rgb, 255:red, 74; green, 144; blue, 226 }  ,draw opacity=1 ]   (356,723.6) -- (297,722.83) ;
\draw [shift={(295,722.8)}, rotate = 0.75] [color={rgb, 255:red, 74; green, 144; blue, 226 }  ,draw opacity=1 ][line width=0.75]    (10.93,-3.29) .. controls (6.95,-1.4) and (3.31,-0.3) .. (0,0) .. controls (3.31,0.3) and (6.95,1.4) .. (10.93,3.29)   ;
\draw    (326,839.8) -- (326,696.8) ;
\draw [shift={(326,694.8)}, rotate = 90] [color={rgb, 255:red, 0; green, 0; blue, 0 }  ][line width=0.75]    (10.93,-3.29) .. controls (6.95,-1.4) and (3.31,-0.3) .. (0,0) .. controls (3.31,0.3) and (6.95,1.4) .. (10.93,3.29)   ;
\draw [color={rgb, 255:red, 74; green, 144; blue, 226 }  ,draw opacity=1 ]   (362.21,812.76) -- (297,811.83) ;
\draw [shift={(295,811.8)}, rotate = 0.81] [color={rgb, 255:red, 74; green, 144; blue, 226 }  ,draw opacity=1 ][line width=0.75]    (10.93,-3.29) .. controls (6.95,-1.4) and (3.31,-0.3) .. (0,0) .. controls (3.31,0.3) and (6.95,1.4) .. (10.93,3.29)   ;

\draw (49.58,771) node [anchor=north west][inner sep=0.75pt]    {$x_{i}$};
\draw (175.42,774.23) node [anchor=north west][inner sep=0.75pt]    {$x_{i+1}$};
\draw (131.35,820) node [anchor=north west][inner sep=0.75pt]    {$x_{k}$};
\draw (252.58,771) node [anchor=north west][inner sep=0.75pt]    {$x_{i}$};
\draw (378.42,775.23) node [anchor=north west][inner sep=0.75pt]    {$x_{i+1}$};
\draw (334.35,816) node [anchor=north west][inner sep=0.75pt]    {$x_{k}$};
\draw (15,861) node [anchor=north west][inner sep=0.75pt]    {$r_{i} :\ \ \ \ \ \ \ \ \ x_{i} x_{k} =x_{k} x_{i}{}_{+}{}_{1} \ \ \ \ \ \ \ \ \ \ \ \ \ \ \ \ \ \ \  \ x_{k} x_{i} =x_{i}{}_{+}{}_{1} x_{k} \ \ \ \ \ $};

\end{tikzpicture}
{\captionsetup{hypcap=false}%

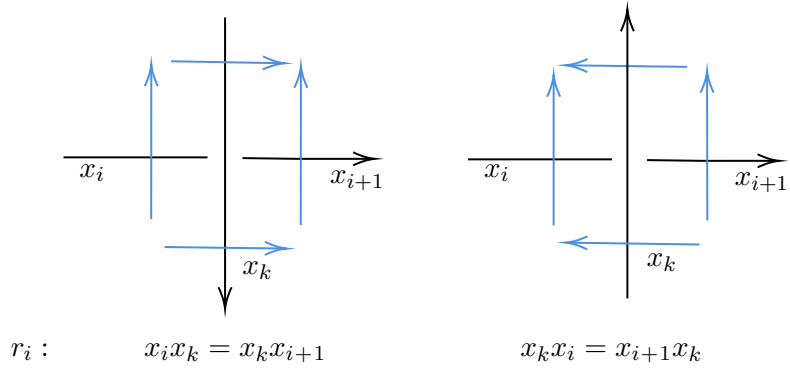
\captionof{figure}{Relator from crossing}%
\label{fig:relator-from-crossing}}

\end{center}

\begin{lemma}\cite{rolfsen2003knots} \label{ref: wirtinger presentation}
    The group $\pi_1(\mathbb{S}^{3}\setminus K)$ is generated by the $($homotopy classes of the$)$ $x_i$ and has presentation
    
    $$\pi_1(\mathbb{S}^{3} \setminus K)=\langle x_1,\dots,x_n\;:\;r_1,\dots,r_n\rangle.$$

   \noindent Moreover, any one of the $r_i$ may be omitted and the above remains true.
\end{lemma}

\subsection{Ideal Triangulation}
Let $K$ be a finite collection of pairwise disjoint Euclidean $3$-simplices, each equipped with its standard simplicial structure. The $2$-simplices of $K$ are referred to as \emph{faces}. Let $\Phi$ be a collection of affine isomorphisms that pair the faces of $K$, subject to the conditions that $\psi \in \Phi$ if and only if $\psi^{-1} \in \Phi$, and that each face of $K$ appears as the domain of exactly one element of $\Phi$. The maps in $\Phi$ are called \emph{face pairings}. Form the quotient space $\widehat{M}$ by identifying faces according to $\Phi$, and let $p \colon K \to \widehat{M}$ denote the associated quotient map. The set of non-manifold points of $\widehat{M}$ is contained in its $0$-skeleton; in this situation, $\widehat{M}$ is referred to as a closed $3$-dimensional pseudomanifold. The triple $\tau=(K,\Phi,p)$ is then called a triangulation of $\widehat{M}$. Let $M_0$ denote the non-compact $3$-manifold obtained from $\widehat{M}$ by deleting all vertices. In this setting, the triangulation $\tau=(K,\Phi,p)$ is called an \emph{ideal} (or \emph{singular}) triangulation of $M_0$.

\subsection{One-Vertex Triangulations and Pachner Moves}\label{pachner_moves}
An important class of triangulations is given by \emph{one-vertex triangulations}, in which all vertices of the tetrahedra are identified to a single point. Such triangulations play a central role in computational topology, as they often provide particularly efficient combinatorial representations of $3$-manifolds. In computational settings, it is common to allow triangulations in which tetrahedra may be geometrically ``bent'' or ``twisted'', so that distinct edges or vertices of the same tetrahedron may be identified. Both theoretical and computational aspects of these triangulations are treated in \cite{jaco20030, matveev2007algorithmic}.

To modify and simplify triangulations without altering the underlying topology, one employs a collection of local transformations known as \emph{Pachner} or \emph{bistellar} moves. There are four elementary Pachner moves, namely the $2$-$3$ and $1$-$4$ moves, along with their reverse. These moves generate an equivalence relation on triangulations and are fundamental both theoretically and computationally. A classical result of Pachner~\cite{pachner1991pl} states that any two triangulations of the same closed $3$-manifold are related by a finite sequence of such moves. Refinements due to Matveev~\cite{matveev2007algorithmic} and Piergallini~\cite{piergallini1988standard} show that any two one-vertex triangulations of a closed $3$-manifold with at least two tetrahedra can be related using only $2$-$3$ and $3$-$2$ moves. The restriction on the number of tetrahedra is essential, since these moves cannot be applied to a single-tetrahedron triangulation.

\section{Faceted \texorpdfstring{\(3\)}{3}-Balls as Building Blocks for the 
\texorpdfstring{\(3\)}{3}-Sphere}\label{construct_manifolds}

We start with a faceted \(3\)-ball given by a tetrahedron with vertices \(A, B, C,\) and \(D\), as shown in Figure \ref{0.vertex-add}. Its boundary consists of four triangular faces. We specify an orientation-preserving face-pairing by identifying these faces in pairs: the face \(CDA\) is paired with the face \(CDB\) by reflection along their common edge \(CD\), and the face \(CAB\) is paired with the face \(DAB\) by reflection along their common edge \(AB\). These define the face-pairing structure of the faceted \(3\)-ball.

To refine this structure, we now subdivide certain boundary edges in a way that is compatible with the chosen face-pairings. Suppose we add a vertex to the edge $CA$. Compatibility with the first face-pairing then requires the addition of a corresponding vertex on the edge \(CB\). Likewise, the second face-pairing forces us to add a vertex on each of the edges \(DA\) and \(DB\), while the edges \(AB\) and \(CD\) remain unchanged. In other words, whenever we insert some vertices on \(CA\), we must insert the same number of vertices on each of the edges \(CB\), \(DA\), and \(DB\).

With this in mind, we subdivide the edges \(CA\), \(CB\), \(DA\), and \(DB\) into \(m+1\) subedges by inserting \(m\) new vertices on each. More explicitly, the edge \(CA\) is subdivided by the vertices $a_m, a_{m-1}, \ldots, a_2, a_1,$ and the edge \(DA\) by $b_m, b_{m-1}, \ldots, b_2, b_1.$ Similarly, the edges \(CB\) and \(DB\) are subdivided by the vertices $c_1, c_2, \ldots, c_m, \; \text{and} \; d_1, d_2, \ldots, d_m,$ respectively; see Figure~\ref{m.vertex-add}. This procedure produces a refined faceted \(3\)-ball, which we will use as a basic building block in our construction of the \(3\)-sphere.

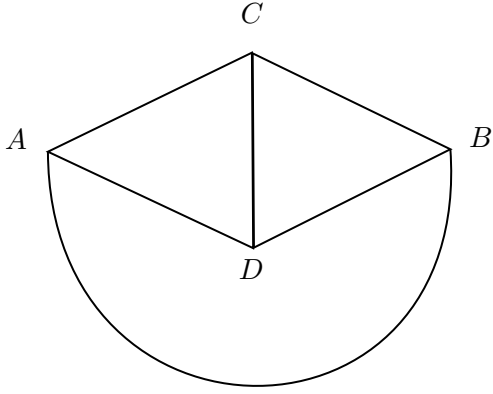
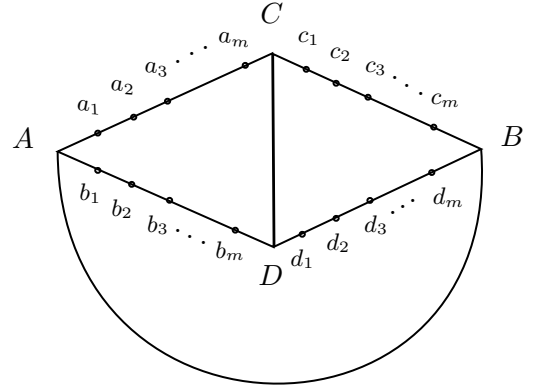
\begin{figure}[!ht]
    \centering
    
\begin{subfigure}[t]{0.45\textwidth}
\centering

\tikzset{every picture/.style={line width=0.75pt}} 

\begin{tikzpicture}[x=0.71pt,y=0.71pt,yscale=-1,xscale=1]

\draw   (355.33,2080.41) -- (463.47,2027.9) -- (464.14,2131.53) -- cycle ;
\draw   (568.68,2079.1) -- (464.53,2131.26) -- (463.88,2028.26) -- cycle ;
\draw    (568.68,2079.1) .. controls (580.35,2245.71) and (356,2246.6) .. (355.33,2080.41) ;

\draw (331,2067) node [anchor=north west][inner sep=0.75pt]    {$A$};
\draw (456,2000) node [anchor=north west][inner sep=0.75pt]    {$C$};
\draw (455,2136) node [anchor=north west][inner sep=0.75pt]    {$D$};
\draw (577,2066) node [anchor=north west][inner sep=0.75pt]    {$B$};

        \end{tikzpicture}
        \caption{Faceted $3$-ball : Tetrahedron}
        \label{0.vertex-add}
    \end{subfigure}%
    \hfill
    \begin{subfigure}[t]{0.45\textwidth}
        \centering

\tikzset{every picture/.style={line width=0.75pt}} 

\begin{tikzpicture}[x=0.75pt,y=0.75pt,yscale=-1,xscale=1]

\draw   (91.25,3178.52) -- (198.96,3129.24) -- (199.63,3226.47) -- cycle ;
\draw   (303.76,3177.28) -- (200.02,3226.22) -- (199.37,3129.58) -- cycle ;
\draw    (303.77,3177.28) .. controls (315.39,3333.61) and (91.91,3334.45) .. (91.24,3178.51) ;
\draw   (214.7,3137.26) .. controls (214.7,3136.59) and (215.28,3136.04) .. (215.99,3136.04) .. controls (216.71,3136.04) and (217.29,3136.59) .. (217.29,3137.26) .. controls (217.29,3137.94) and (216.71,3138.48) .. (215.99,3138.48) .. controls (215.28,3138.48) and (214.7,3137.94) .. (214.7,3137.26) -- cycle ;
\draw   (110.1,3169.26) .. controls (110.1,3168.59) and (110.68,3168.04) .. (111.4,3168.04) .. controls (112.11,3168.04) and (112.69,3168.59) .. (112.69,3169.26) .. controls (112.69,3169.94) and (112.11,3170.48) .. (111.4,3170.48) .. controls (110.68,3170.48) and (110.1,3169.94) .. (110.1,3169.26) -- cycle ;
\draw   (110.1,3187.99) .. controls (110.1,3187.32) and (110.68,3186.77) .. (111.4,3186.77) .. controls (112.11,3186.77) and (112.69,3187.32) .. (112.69,3187.99) .. controls (112.69,3188.67) and (112.11,3189.21) .. (111.4,3189.21) .. controls (110.68,3189.21) and (110.1,3188.67) .. (110.1,3187.99) -- cycle ;
\draw   (212.69,3220.05) .. controls (212.69,3219.38) and (213.27,3218.83) .. (213.99,3218.83) .. controls (214.7,3218.83) and (215.28,3219.38) .. (215.28,3220.05) .. controls (215.28,3220.73) and (214.7,3221.27) .. (213.99,3221.27) .. controls (213.27,3221.27) and (212.69,3220.73) .. (212.69,3220.05) -- cycle ;
\draw   (128.1,3161.26) .. controls (128.1,3160.59) and (128.68,3160.04) .. (129.4,3160.04) .. controls (130.11,3160.04) and (130.69,3160.59) .. (130.69,3161.26) .. controls (130.69,3161.94) and (130.11,3162.48) .. (129.4,3162.48) .. controls (128.68,3162.48) and (128.1,3161.94) .. (128.1,3161.26) -- cycle ;
\draw   (184.1,3135.26) .. controls (184.1,3134.59) and (184.68,3134.04) .. (185.4,3134.04) .. controls (186.11,3134.04) and (186.69,3134.59) .. (186.69,3135.26) .. controls (186.69,3135.94) and (186.11,3136.48) .. (185.4,3136.48) .. controls (184.68,3136.48) and (184.1,3135.94) .. (184.1,3135.26) -- cycle ;
\draw   (145.1,3153.26) .. controls (145.1,3152.59) and (145.68,3152.04) .. (146.4,3152.04) .. controls (147.11,3152.04) and (147.69,3152.59) .. (147.69,3153.26) .. controls (147.69,3153.94) and (147.11,3154.48) .. (146.4,3154.48) .. controls (145.68,3154.48) and (145.1,3153.94) .. (145.1,3153.26) -- cycle ;
\draw   (127.1,3194.99) .. controls (127.1,3194.32) and (127.68,3193.77) .. (128.4,3193.77) .. controls (129.11,3193.77) and (129.69,3194.32) .. (129.69,3194.99) .. controls (129.69,3195.67) and (129.11,3196.21) .. (128.4,3196.21) .. controls (127.68,3196.21) and (127.1,3195.67) .. (127.1,3194.99) -- cycle ;
\draw   (146.1,3202.99) .. controls (146.1,3202.32) and (146.68,3201.77) .. (147.4,3201.77) .. controls (148.11,3201.77) and (148.69,3202.32) .. (148.69,3202.99) .. controls (148.69,3203.67) and (148.11,3204.21) .. (147.4,3204.21) .. controls (146.68,3204.21) and (146.1,3203.67) .. (146.1,3202.99) -- cycle ;
\draw   (179.1,3217.99) .. controls (179.1,3217.32) and (179.68,3216.77) .. (180.4,3216.77) .. controls (181.11,3216.77) and (181.69,3217.32) .. (181.69,3217.99) .. controls (181.69,3218.67) and (181.11,3219.21) .. (180.4,3219.21) .. controls (179.68,3219.21) and (179.1,3218.67) .. (179.1,3217.99) -- cycle ;
\draw   (229.7,3144.26) .. controls (229.7,3143.59) and (230.28,3143.04) .. (230.99,3143.04) .. controls (231.71,3143.04) and (232.29,3143.59) .. (232.29,3144.26) .. controls (232.29,3144.94) and (231.71,3145.48) .. (230.99,3145.48) .. controls (230.28,3145.48) and (229.7,3144.94) .. (229.7,3144.26) -- cycle ;
\draw   (245.7,3151.26) .. controls (245.7,3150.59) and (246.28,3150.04) .. (246.99,3150.04) .. controls (247.71,3150.04) and (248.29,3150.59) .. (248.29,3151.26) .. controls (248.29,3151.94) and (247.71,3152.48) .. (246.99,3152.48) .. controls (246.28,3152.48) and (245.7,3151.94) .. (245.7,3151.26) -- cycle ;
\draw   (278.7,3166.26) .. controls (278.7,3165.59) and (279.28,3165.04) .. (279.99,3165.04) .. controls (280.71,3165.04) and (281.29,3165.59) .. (281.29,3166.26) .. controls (281.29,3166.94) and (280.71,3167.48) .. (279.99,3167.48) .. controls (279.28,3167.48) and (278.7,3166.94) .. (278.7,3166.26) -- cycle ;
\draw   (229.69,3212.05) .. controls (229.69,3211.38) and (230.27,3210.83) .. (230.99,3210.83) .. controls (231.7,3210.83) and (232.28,3211.38) .. (232.28,3212.05) .. controls (232.28,3212.73) and (231.7,3213.27) .. (230.99,3213.27) .. controls (230.27,3213.27) and (229.69,3212.73) .. (229.69,3212.05) -- cycle ;
\draw   (246.69,3203.05) .. controls (246.69,3202.38) and (247.27,3201.83) .. (247.99,3201.83) .. controls (248.7,3201.83) and (249.28,3202.38) .. (249.28,3203.05) .. controls (249.28,3203.73) and (248.7,3204.27) .. (247.99,3204.27) .. controls (247.27,3204.27) and (246.69,3203.73) .. (246.69,3203.05) -- cycle ;
\draw   (277.69,3189.05) .. controls (277.69,3188.38) and (278.27,3187.83) .. (278.99,3187.83) .. controls (279.7,3187.83) and (280.28,3188.38) .. (280.28,3189.05) .. controls (280.28,3189.73) and (279.7,3190.27) .. (278.99,3190.27) .. controls (278.27,3190.27) and (277.69,3189.73) .. (277.69,3189.05) -- cycle ;

\draw (66.97,3165.25) node [anchor=north west][inner sep=0.75pt]    {$A$};
\draw (191.49,3102.38) node [anchor=north west][inner sep=0.75pt]    {$C$};
\draw (190.49,3234) node [anchor=north west][inner sep=0.75pt]    {$D$};
\draw (312.03,3164.31) node [anchor=north west][inner sep=0.75pt]    {$B$};
\draw (99.23,3151) node [anchor=north west][inner sep=0.75pt]  [font=\footnotesize,rotate=-4.48,xslant=-0.12]  {$a_{1}$};
\draw (100.71,3193) node [anchor=north west][inner sep=0.75pt]  [font=\footnotesize]  {$b_{1}$};
\draw (210.27,3117) node [anchor=north west][inner sep=0.75pt]  [font=\footnotesize]  {$c_{1}$};
\draw (206.65,3226) node [anchor=north west][inner sep=0.75pt]  [font=\footnotesize]  {$d_{1}$};
\draw (116.23,3141) node [anchor=north west][inner sep=0.75pt]  [font=\footnotesize,rotate=-4.48,xslant=-0.12]  {$a_{2}$};
\draw (133.23,3132) node [anchor=north west][inner sep=0.75pt]  [font=\footnotesize,rotate=-4.48,xslant=-0.12]  {$a_{3}$};
\draw (170.23,3117) node [anchor=north west][inner sep=0.75pt]  [font=\footnotesize,rotate=-4.48,xslant=-0.12]  {$a_{m}$};
\draw (116.71,3201) node [anchor=north west][inner sep=0.75pt]  [font=\footnotesize]  {$b_{2}$};
\draw (134.71,3209) node [anchor=north west][inner sep=0.75pt]  [font=\footnotesize]  {$b_{3}$};
\draw (169,3222) node [anchor=north west][inner sep=0.75pt]  [font=\footnotesize]  {$b_{m}$};
\draw (226.27,3125) node [anchor=north west][inner sep=0.75pt]  [font=\footnotesize]  {$c_{2}$};
\draw (244,3132) node [anchor=north west][inner sep=0.75pt]  [font=\footnotesize]  {$c_{3}$};
\draw (277,3148) node [anchor=north west][inner sep=0.75pt]  [font=\footnotesize]  {$c_{m}$};
\draw (224.65,3218) node [anchor=north west][inner sep=0.75pt]  [font=\footnotesize]  {$d_{2}$};
\draw (243.65,3209) node [anchor=north west][inner sep=0.75pt]  [font=\footnotesize]  {$d_{3}$};
\draw (278,3195) node [anchor=north west][inner sep=0.75pt]  [font=\footnotesize]  {$d_{m}$};
\draw (148,3216) node [anchor=north west][inner sep=0.75pt]   [align=left] {.};
\draw (154,3219) node [anchor=north west][inner sep=0.75pt]   [align=left] {.};
\draw (161,3222) node [anchor=north west][inner sep=0.75pt]   [align=left] {.};
\draw (256,3207) node [anchor=north west][inner sep=0.75pt]   [align=left] {.};
\draw (261,3204) node [anchor=north west][inner sep=0.75pt]   [align=left] {.};
\draw (266,3201) node [anchor=north west][inner sep=0.75pt]   [align=left] {.};
\draw (257,3139) node [anchor=north west][inner sep=0.75pt]   [align=left] {.};
\draw (263,3143) node [anchor=north west][inner sep=0.75pt]   [align=left] {.};
\draw (270,3147) node [anchor=north west][inner sep=0.75pt]   [align=left] {.};
\draw (162,3124) node [anchor=north west][inner sep=0.75pt]   [align=left] {.};
\draw (155,3127) node [anchor=north west][inner sep=0.75pt]   [align=left] {.};
\draw (148,3131) node [anchor=north west][inner sep=0.75pt]   [align=left] {.};

\end{tikzpicture}

        \caption{Inserting $m$ vertices in tetrahedron}
        \label{m.vertex-add}
    \end{subfigure}

    \caption{Cellulation for adding vertices to a tetrahedron}
    \label{add_vertex_tetrahedra}
\end{figure}

We now identify the subdivided faces via vertex homeomorphisms. More precisely,
the subdivided face \(CDA\) is identified with the subdivided face \(CDB\) by
reflection along their common edge \(CD\), inducing a vertex homeomorphism
\(\epsilon_{1,m}\). Similarly, the subdivided face \(CAB\) is identified with
the subdivided face \(DAB\) by a twist along one subedge, giving rise to a
second vertex homeomorphism \(\epsilon_{2,m}\). For clarity, we denote these
identifications as follows:
\[
\epsilon_{1,m} :
\begin{pmatrix}
C & D & b_m & \cdots & b_1 & A & a_1 & a_2 & \cdots & a_m \\
C & D & d_1 & \cdots & d_m & B & c_m & c_{m-1} & \cdots & c_1
\end{pmatrix}
\]

\[
\epsilon_{2,m} :
\begin{pmatrix}
C & a_m & \cdots & a_1 & A & B & c_m & \cdots & c_2 & c_1 \\
b_m & b_{m-1} & \cdots & A & B & d_m & d_{m-1} & \cdots & d_1 & D
\end{pmatrix}.
\]

The maps \(\epsilon_{1,m}\) and \(\epsilon_{2,m}\) are called
\emph{face-identification maps}, and we write the collection of all such identification maps
\(\epsilon_m = \{\epsilon_{1,m}^{\pm 1}, \epsilon_{2,m}^{\pm 1}\}\).

\smallskip

By performing the face identifications prescribed by the maps \(\epsilon_m\) on the cell complex depicted in Figure~\ref{m.vertex-add}, we obtain a family of closed \(3\)-dimensional pseudomanifolds, denoted by \(M(\epsilon_m)\). We will later show that each member of this family is, in fact, a \(3\)-manifold (cf. Lemma \ref{$3$-manifold}). Before establishing this result, we first prove the following lemma, which provides a key ingredient for the arguments that follow. Throughout, we fix the following notation: $m$ denotes the number of inserted vertices, with $m \in \mathbb{Z}_{\geq0}$; $\epsilon_m$ denotes the orientation-reversing face-identification maps; and $M(\epsilon_m)$ denotes the quotient space obtained under these identifications.


\begin{lemma}\label{lemma:two_edge}
If $\epsilon_m$ is an orientation-reversing face-identification map on the modified tetrahedron as defined above, then for every non-negative integer $m$, the quotient space $M(\epsilon_m)$ has exactly one vertex and two edges.
\end{lemma}
\begin{proof}
Let $m \in \mathbb{Z}_{\geq0}$ denote the number of vertices inserted along each of the edges $AC$, $AD$, $CB$, and $DB$, we call them subdivided edges, and let $\epsilon_m$ be the associated collection of face-identification maps defined in the previous paragraph. To prove the lemma, we analyze the resulting quotient space $M(\epsilon_m)$ by distinguishing two cases according to whether $m$ is even or odd. We first consider the case where $m$ is odd, and subsequently treat the case where $m$ is even.

\begin{figure}[ht]
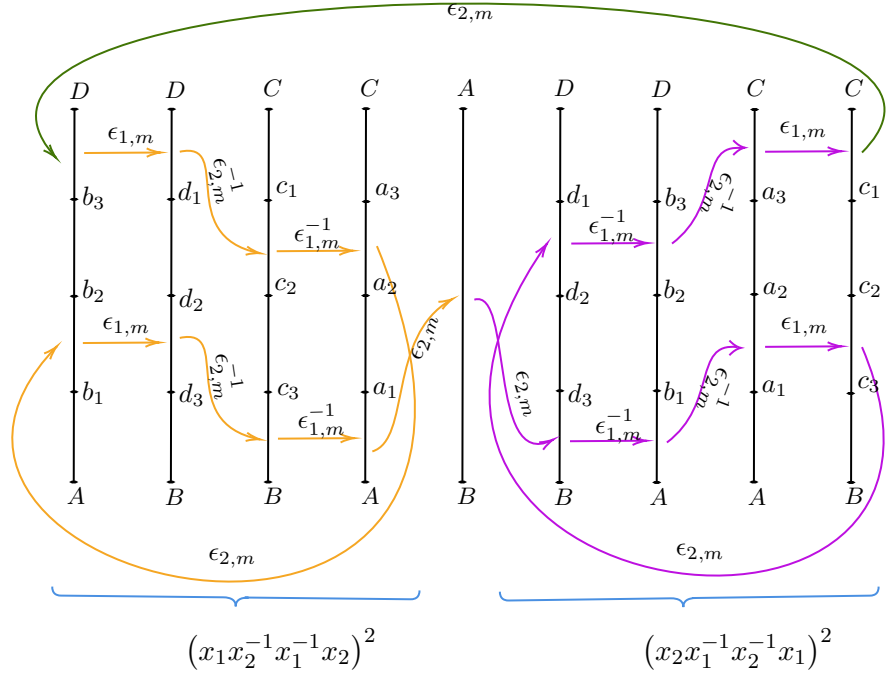

\centering
\tikzset{every picture/.style={line width=0.75pt}} 


\vspace{-1cm}
\captionof{figure}{For $m=3$ second edge cycle}%
\label{fig:edgecycle}
\end{figure}

\begin{description}
\item[\underline{Case~1} (Adding an odd number of vertices i.e. $m$ is odd).]
Suppose that an odd number of vertices is inserted along each of the subdivided edges (see in Figure \ref{m.vertex-add}). Let $m=2t+1$ for $t\in \mathbb{N}\cup \{0\}$. In this situation, each subdivided edge consists of an even number of subedges. Using the face-identification maps $\epsilon_m$, we examine the edge cycles in the quotient space $M(\epsilon_m)$. Observe first that the edge $CD$ is identified with itself under the map $\epsilon_{1,m}$, yielding one edge of $M(\epsilon_m)$. To determine another edge, we begin with the edge $AB$. Under the identification $\epsilon_{2,m}$, the edge $AB$ maps to the final subedge of $DB$, denoted by $B d_m$. Applying $\epsilon_{1,m}^{-1}$, the subedge $B d_m$ is then identified with the first subedge of $AD$, namely $A b_1$. Next, via $\epsilon_{2,m}^{-1}$, this subedge maps to the second subedge of $AC$, given by $a_1 a_2$. Finally, under $\epsilon_{1,m}$, the subedge $a_1 a_2$ maps to the last second subedge of $CB$, namely $c_m c_{m-1}$. This completes the first step of the edge traversal induced by the composition of maps \(\epsilon_2 \epsilon_1^{-1} \epsilon_2^{-1} \epsilon_1\).

In the second step, the subedge $c_m c_{m-1}$ of $CB$ maps to the third-to-last subedge of $DB$, namely $d_{m-1} d_{m-2}$, via $\epsilon_{2,m}$. Applying $\epsilon_{1,m}^{-1}$, this subedge is identified with the third subedge of $AD$, denoted by $b_2 b_3$. Subsequently, under $\epsilon_{2,m}^{-1}$, it maps to the fourth subedge of $AC$, namely $a_3 a_4$, and then to the fourth-to-last subedge of $CB$, given by $c_{m-2} c_{m-3}$, via $\epsilon_{1,m}$. This completes the second step, in which the face-pairing maps also follow the composition \(\epsilon_2 \epsilon_1^{-1} \epsilon_2^{-1} \epsilon_1\).

Proceeding inductively in this manner, the traversal continues along the subdivided edges according to the same sequence of face identifications. It is important to observe the following $(+2)$ subedge rule: along any fixed subdivided edge, if a subedge (for example, the first) enters an edge cycle at a given step of the traversal, then after one complete step the adjacent subedge does not enter the cycle; rather, the subedge two positions ahead along the same subdivided edge (for example, the third) enters the edge cycle at the next step. As a result, exactly half of the subedges of the subdivided edges form a closed loop via the composition of maps $(\epsilon_2 \epsilon_1^{-1} \epsilon_2^{-1} \epsilon_1)^{t+1}$, eventually returning to the first subedge $c_1C$ of $CB$. This completes the first closed loop of the second edge cycle.

From $c_1C$, the traversal proceeds to the final subedge of $AD$, namely $D b_m$, via $\epsilon_{2,m}$. Then the edge \(Db_m\) is mapped to the first subedge of \(DB\), namely \(Dd_1\), via the map \(\epsilon_{1,m}\), that initiates the second closed loop of the edge cycle, which again follows the same $(+2)$ subedge rule: after one subedge enters the cycle, every second subedge thereafter is included. Continuing in this way, the cycle eventually reaches the first subedge of $AC$, namely $a_1A$, which is then identified with the edge $AB$ via $\epsilon_{2,m}$. In this second closed loop of the second edge cycle, the remaining half of the subedges of the subdivided edges follow the face-identification maps given by $(t+1)$ times $(\epsilon_{1,m}\epsilon_{2,m}^{-1}\epsilon_{1,m}^{-1}\epsilon_{2,m})$. This completes the description of the second edge cycle in $M(\epsilon_m)$.

Therefore, all subedges of the subdivided edges, together with the original edges $AB$ and $CD$, are partitioned into exactly two edge cycles in the quotient space $M(\epsilon_m)$, where $m$ is odd. In short, the first edge formed via $\epsilon_{1,m}$ map and the second edge formed via the composition of maps $(\epsilon_2 \epsilon_1^{-1} \epsilon_2^{-1} \epsilon_1)^{t+1} \; \epsilon_{2,m}\; (\epsilon_{1,m}\epsilon_{2,m}^{-1}\epsilon_{1,m}^{-1}\epsilon_{2,m})^{t+1}$. For example, we illustrate the second edge cycle in the case $m=3$ in Figure \ref{fig:edgecycle}.

\end{description}

\begin{figure}[ht]
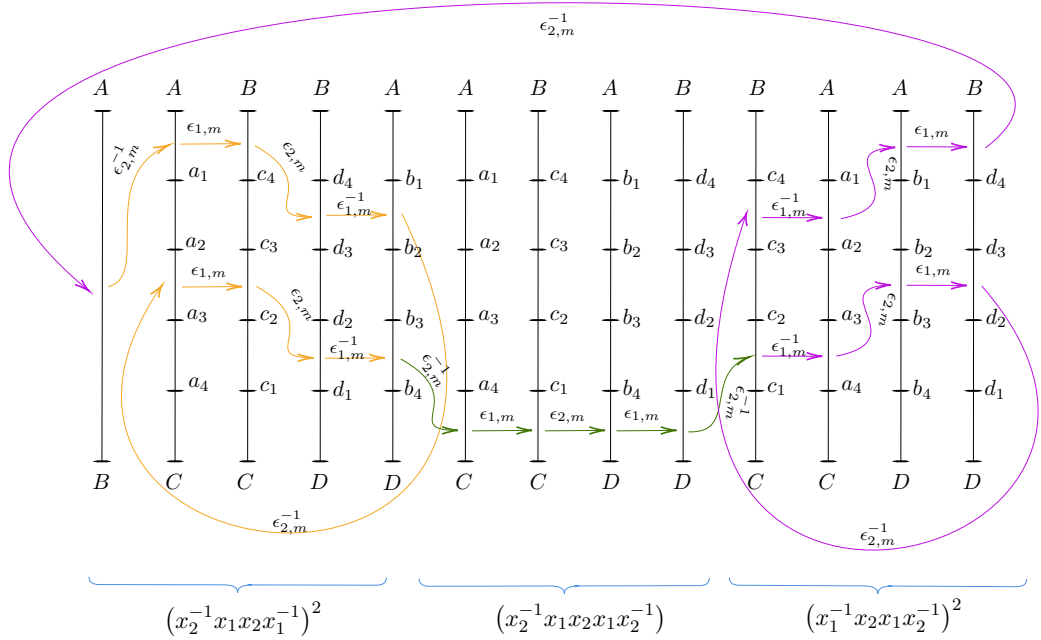

\centering
\hspace*{-0.23\textwidth} 
\resizebox{1.3\textwidth}{!}{%

}
\vspace{-1.2cm}
\caption{For $m=4$, the second edge cycle}
\label{fig:4.edgecycle}
\end{figure}

\begin{description}
\item[\underline{Case~2} (Adding an even number of vertices i.e. $m$ is even).]
We now consider the case in which an even number of vertices is inserted along each of the subdivided edges (see in Figure~\ref{m.vertex-add}). Let $m=2t$ for $t\in \mathbb{N}\cup \{0\}$. In this situation, each subdivided edge consists of an odd number of subedges. Now we examine the edge cycles in the quotient space $M(\epsilon_m)$.

As in the previous case, the edge $CD$ is identified with itself under the map $\epsilon_{1,m}$, yielding one edge of $M(\epsilon_m)$. To determine another edge cycle, we begin with the edge $AB$ and follow its images under the face-identification maps.

In the first step, the edge $AB$ maps to the first subedge of $AC$, denoted by $A a_1$, via the map $\epsilon_{2,m}^{-1}$. This subedge then maps to the final subedge of $CB$, denoted by $B c_m$, under $\epsilon_{1,m}$. Next, $B c_m$ maps to the second-to-last subedge of $DB$, namely $d_m d_{m-1}$, via $\epsilon_{2,m}$. Applying $\epsilon_{1,m}^{-1}$, this subedge is identified with the second subedge of $AD$, denoted by $b_1 b_2$. This completes the first step of the edge traversal induced by the composition $\epsilon_{2,m}^{-1}\epsilon_{1,m}\epsilon_{2,m}\epsilon_{1,m}^{-1}$.

The edge traversal then continues through successive subedges of $AC$, following the $(+2)$-subedge rule, and proceeds step by step along the subdivided edges. This loop is completed when the edge cycle reaches the second-to-last subedge of $AD$, denoted by $b_{m-1} b_{m}$. This first closed loop follow the face-identification maps given by $t$ times $\epsilon_{2,m}^{-1}\epsilon_{1,m}\epsilon_{2,m}\epsilon_{1,m}^{-1}$. 

From the subedge $b_{m-1} b_m$, the traversal proceeds to the final subedge of $AC$, denoted by $a_m C$, via $\epsilon_{2,m}^{-1}$. This subedge then maps to the first subedge of $CB$, namely $c_1C$, under $\epsilon_{1,m}$. Next, $c_1C$ maps to the final subedge of $AD$, denoted by $Db_m$, via $\epsilon_{2,m}$. Applying $\epsilon_{1,m}$, the subedge $Db_m$ is identified with the first subedge of $DB$, namely $Dd_1$. Finally, $D d_1$ maps to the second subedge of $CB$, denoted by $c_1 c_2$, via $\epsilon_{2,m}^{-1}$. This completes a second loop via the composition map $\epsilon_{2,m}^{-1}\epsilon_{1,m}\epsilon_{2,m}\epsilon_{1,m}\epsilon_{2,m}^{-1}$, which occurs only once and terminates at the subedge $c_1 c_2$.

Then the subedge $c_1 c_2$ then maps to the second-to-last subedge of $AC$, namely $a_m a_{m-1}$, via $\epsilon_{1,m}^{-1}$. Under $\epsilon_{2,m}$, this subedge maps to the third-to-last subedge of $AD$, denoted by $b_{m-1} b_{m-2}$. Applying $\epsilon_{1,m}$, it is then identified with the third subedge of $DB$, namely $d_2 d_3$. This completes the first step of the final edge traversal, which continues inductively following the $(+2)$-subedge rule. Eventually, the traversal reaches the final subedge $d_m B$ of $DB$, which maps to the edge $AB$ via $\epsilon_{2,m}^{-1}$. Thus, the final loop of the second edge cycle follows the face-identification maps $\epsilon_{1,m}^{-1}\epsilon_{2,m}\epsilon_{1,m}\epsilon_{2,m}^{-1}$ repeated \(t\) times. This completes the description of the second edge cycle of \(M(\epsilon_m)\).

Consequently, all subedges of the subdivided edges, together with the original edges $AB$ and $CD$, are partitioned into exactly two edge cycles in the quotient space $M(\epsilon_m)$. In summary, the first edge generated via $\epsilon_{1,m}$ and the second edge formed via the composition of maps followed by $(\epsilon_{2,m}^{-1}\epsilon_{1,m}\epsilon_{2,m}\epsilon_{1,m}^{-1})^{t}\;(\epsilon_{2,m}^{-1}\epsilon_{1,m}\epsilon_{2,m}\epsilon_{1,m}\epsilon_{2,m}^{-1})(\epsilon_{1,m}^{-1}\epsilon_{2,m}\epsilon_{1,m}\epsilon_{2,m}^{-1})^{t}$. For illustration, we depict the second edge cycle in the case $m = 4$ in Figure~\ref{fig:4.edgecycle}.

\end{description}  

From the above argument, all vertices of the faceted $3$-ball are identified to a single vertex  under the face identification $\epsilon_m$. Hence, the quotient space $M(\epsilon_m)$ has exactly one vertex and two edges.
\end{proof}

\begin{lemma}\label{$3$-manifold}
For any non-negative integer \(m\), let \(M(\epsilon_m)\) be defined as above. Then \(M(\epsilon_m)\) is a closed \(3\)-manifold. 
\end{lemma}

\begin{proof}
It is obvious that $M(\epsilon_m)$ is a closed, orientable $3$-dimensional pseudomanifold. Let $f_i$ denote the number of $i$-cells in $M(\epsilon_m)$. By Lemma~\ref{lemma:two_edge}, the resulting cell structure has the following cell counts: $f_0 = 1$, $f_1 = 2$, $f_2 = 2$, and $f_3 = 1$. Hence, the Euler characteristic of $M(\epsilon_m)$ is
\[
\chi\bigl(M(\epsilon_m)\bigr) = f_0 - f_1 + f_2 - f_3 = 1 - 2 + 2 - 1 = 0.
\]
Therefore, by Lemma \ref{pm-m}, it follows that $M(\epsilon_m)$ is a closed, orientable $3$-manifold.
\end{proof}

Thus, applying the face identifications $\epsilon_m$ to Figure~\ref{m.vertex-add}  yields a family of \(3\)-manifolds arising from this construction. The fundamental group of each resulting manifold can be calculated directly from the associated edge-cycle diagrams. For that, let \(x_1\) and \(x_2\) denote generators of a free group, and relators are then assigned by replacing each occurrence of the symbols \(\epsilon_{i,m}\) in the edge-cycle diagrams with the corresponding generators \(x_i\). Consequently, the generators are determined by the face identifications used in the construction, while the relators reflect the combinatorial structure encoded in the edge-cycle diagrams. It is well known that if $M$ is a closed $3$-manifold with exactly one vertex and $M_0$ denotes the non-compact manifold obtained by removing that vertex (that is, the open manifold corresponding to the complement of the $0$-skeleton), then by the Seifert–Van Kampen theorem, $M$ and $M_0$ have the same fundamental group. This observation allows us to compute the fundamental group directly from the face-identification data without further modification. The classical theorem is stated as follows:

\begin{lemma}\cite{cannon2000introduction}\label{Th:f.g}
The closed $3$-manifold $M$ and the non-compact manifold $M_0$ have same isomorphic fundamental groups, admitting the presentation
\[
\langle x_1, x_2, \ldots \mid W_1, W_2, \ldots \rangle .
\]
Here, the generators $x_1, x_2, \ldots$ correspond to the face-identification maps, while the relators $W_1, W_2,$
$\ldots$ are the words determined by the edge cycles described above.
\end{lemma}

\begin{theorem}\label{3-sphere}
For each non-negative integer $m$, let $M(\epsilon_m)$ be the $3$-manifold defined above.  
Then the fundamental group $\pi_1(M(\epsilon_m))$ admits a presentation with two generators
$x_1, x_2$ and two relations, given as follows.

\begin{itemize}
\item If $m$ is even, say $m = 2t$ for some $t \in \mathbb{N} \cup \{0\}$, then
\begin{equation}\label{f.g.1}
\pi_1(M(\epsilon_m)) =
\left\langle x_1, x_2 \;\middle|\;
x_1,\;
(x_2^{-1} x_1 x_2 x_1^{-1})^{t}
(x_2^{-1}x_1x_2x_1x_2^{-1})
(x_1^{-1}x_2x_1x_2^{-1})^{t}
\right\rangle.
\end{equation}

\item If $m$ is odd, say $m = 2t + 1$ for some $t \in \mathbb{N} \cup \{0\}$, then
\begin{equation}\label{f.g.2}
\pi_1(M(\epsilon_m)) =
\left\langle x_1, x_2 \;\middle|\;
x_1,\;
(x_2 x_1^{-1} x_2^{-1} x_1)^{t+1}
(x_2)
(x_1 x_2^{-1} x_1^{-1} x_2)^{t+1}
\right\rangle.
\end{equation}
\end{itemize}

Consequently, the manifold $M(\epsilon_m)$ is a $3$-sphere $\mathbb{S}^3$.
\end{theorem}

\begin{proof}
By Lemma~\ref{Th:f.g}, the fundamental group of $M(\epsilon_m)$ is obtained directly from its edge cycles. As shown in the proof of Lemma~\ref{lemma:two_edge}, these cycles depend on the choice of the non-negative integer $m$. After replacing the symbols $\epsilon_{i,m}$ by the generators $x_i$, the first edge cycle always yields the relator $x_1$, while the second relator arises from the second edge cycle and satisfies the relations in \cref{f.g.1} and \cref{f.g.2}, varying with the choice of \(m\). In this way, the presentation of the fundamental group $\pi_1(M(\epsilon_m))$ follows naturally from the combinatorial structure of the edge cycles.

From the first relation, we have $x_1 = 1$. Substituting this into the second relation immediately implies that $x_2 = 1$ as well. Thus, the fundamental group $\pi_1(M(\epsilon_m))$ is trivial, and hence, by Proposition~\ref{propo:sphere}, the manifold $M(\epsilon_m)$ is a $3$-sphere.
\end{proof}

\section{Topological Construction of Twist Knot Complement}\label{twist.knot.complement}
In this section, we present a topological construction of the twist knot complement. To achieve this, we modify our faceted $3$-ball (a tetrahedron) by collapsing the edge $CD$ to a single vertex $C$, in each case where the resulting space is a $3$-sphere. The modified faceted $3$-ball will be referred to as the \emph{collapsed faceted $3$-ball}. See Figure \ref{collapsed faceted $3$-ball}.

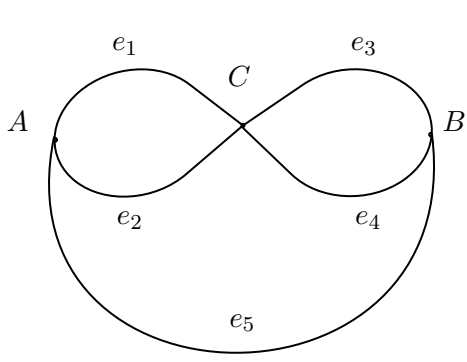
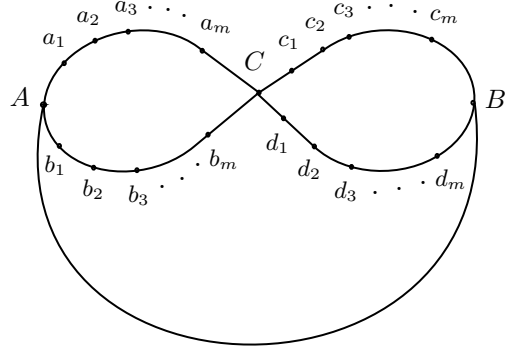
\begin{figure}[!ht]
    \centering
    
    \begin{subfigure}[t]{0.45\textwidth}
        \centering
   \tikzset{every picture/.style={line width=0.75pt}}
        \begin{tikzpicture}[scale=1pt, x=0.77pt, y=0.77pt, yscale=-1, xscale=1]
            
\draw   (426.18,1818.06) .. controls (426.18,1818.06) and (426.18,1818.06) .. (426.18,1818.06) .. controls (426.18,1818.06) and (426.18,1818.06) .. (426.18,1818.06) .. controls (410.84,1831.19) and (386.69,1832.83) .. (372.24,1821.74) .. controls (357.78,1810.64) and (358.48,1791.01) .. (373.82,1777.88) .. controls (389.15,1764.75) and (413.3,1763.11) .. (427.76,1774.2) .. controls (445.22,1787.59) and (453.94,1794.28) .. (453.94,1794.29) .. controls (453.94,1794.28) and (444.68,1802.21) .. (426.18,1818.06) -- cycle ;
\draw   (483.15,1774.43) .. controls (499.69,1763.18) and (524.03,1764.54) .. (537.53,1777.46) .. controls (551.03,1790.38) and (548.58,1809.98) .. (532.04,1821.23) .. controls (515.51,1832.48) and (491.16,1831.12) .. (477.66,1818.19) .. controls (477.66,1818.19) and (477.66,1818.19) .. (477.66,1818.19) .. controls (461.36,1802.6) and (453.21,1794.8) .. (453.22,1794.79) .. controls (453.21,1794.8) and (463.2,1788) .. (483.15,1774.43) -- cycle ;
\draw    (361.63,1799.62) .. controls (332.69,1939.12) and (564.25,1942.48) .. (546.33,1798.78) ;
\draw   (453.26,1794.29) .. controls (453.26,1793.83) and (453.57,1793.45) .. (453.95,1793.45) .. controls (454.33,1793.46) and (454.64,1793.83) .. (454.64,1794.3) .. controls (454.64,1794.76) and (454.32,1795.14) .. (453.94,1795.14) .. controls (453.56,1795.14) and (453.25,1794.76) .. (453.26,1794.29) -- cycle ;
\draw   (361.63,1801.3) .. controls (361.63,1800.83) and (361.94,1800.46) .. (362.32,1800.46) .. controls (362.7,1800.46) and (363.01,1800.84) .. (363.01,1801.31) .. controls (363.01,1801.77) and (362.7,1802.15) .. (362.32,1802.15) .. controls (361.94,1802.14) and (361.63,1801.77) .. (361.63,1801.3) -- cycle ;
\draw   (544.95,1798.77) .. controls (544.95,1798.31) and (545.27,1797.93) .. (545.65,1797.94) .. controls (546.03,1797.94) and (546.34,1798.32) .. (546.33,1798.78) .. controls (546.33,1799.25) and (546.02,1799.62) .. (545.64,1799.62) .. controls (545.26,1799.62) and (544.95,1799.24) .. (544.95,1798.77) -- cycle ;

\draw (445.14,1763.75) node [anchor=north west][inner sep=0.75pt]    {$C$};
\draw (337,1785.6) node [anchor=north west][inner sep=0.75pt]    {$A$};
\draw (550,1785.6) node [anchor=north west][inner sep=0.75pt]    {$B$};
\draw (388.47,1750) node [anchor=north west][inner sep=0.75pt]    {$e_{1}$};
\draw (390.85,1835) node [anchor=north west][inner sep=0.75pt]    {$e_{2}$};
\draw (507.32,1835) node [anchor=north west][inner sep=0.75pt]    {$e_{4}$};
\draw (505.25,1750) node [anchor=north west][inner sep=0.75pt]    {$e_{3}$};
\draw (445.85,1885) node [anchor=north west][inner sep=0.75pt]    {$e_{5}$};

        \end{tikzpicture}
        \caption{Collapsed Faceted $3$-ball without subedge}
        \label{0.faceted_cell}
    \end{subfigure}%
    \hfill
    \begin{subfigure}[t]{0.45\textwidth}
        \centering

\tikzset{every picture/.style={line width=0.75pt}} 

\begin{tikzpicture}[x=0.75pt,y=0.75pt,yscale=-1,xscale=1]

\draw   (456.59,3190.98) .. controls (456.59,3190.98) and (456.59,3190.98) .. (456.59,3190.98) .. controls (456.59,3190.98) and (456.59,3190.98) .. (456.59,3190.98) .. controls (438.63,3205.92) and (410.33,3207.8) .. (393.39,3195.17) .. controls (376.44,3182.54) and (377.27,3160.19) .. (395.24,3145.25) .. controls (413.2,3130.31) and (441.5,3128.43) .. (458.45,3141.06) .. controls (458.45,3141.06) and (458.45,3141.06) .. (458.45,3141.06) .. controls (478.9,3156.3) and (489.13,3163.93) .. (489.12,3163.93) .. controls (489.13,3163.93) and (478.29,3172.94) .. (456.59,3190.98) -- cycle ;
\draw   (523.35,3141.32) .. controls (523.35,3141.32) and (523.35,3141.32) .. (523.35,3141.32) .. controls (542.72,3128.52) and (571.24,3130.06) .. (587.05,3144.77) .. controls (602.86,3159.48) and (599.98,3181.79) .. (580.62,3194.59) .. controls (561.25,3207.39) and (532.73,3205.85) .. (516.92,3191.14) .. controls (497.83,3173.38) and (488.29,3164.5) .. (488.28,3164.5) .. controls (488.29,3164.5) and (499.98,3156.77) .. (523.35,3141.32) -- cycle ;
\draw    (380.97,3170) .. controls (347.06,3328.78) and (618.36,3332.6) .. (597.36,3169.04) ;
\draw   (390.53,3149.19) .. controls (390.54,3148.66) and (390.9,3148.23) .. (391.35,3148.24) .. controls (391.79,3148.24) and (392.15,3148.67) .. (392.15,3149.2) .. controls (392.15,3149.73) and (391.79,3150.15) .. (391.34,3150.15) .. controls (390.89,3150.15) and (390.53,3149.72) .. (390.53,3149.19) -- cycle ;
\draw   (388.13,3190.71) .. controls (388.13,3190.18) and (388.49,3189.75) .. (388.94,3189.75) .. controls (389.39,3189.76) and (389.75,3190.19) .. (389.74,3190.72) .. controls (389.74,3191.25) and (389.38,3191.67) .. (388.93,3191.67) .. controls (388.49,3191.67) and (388.12,3191.24) .. (388.13,3190.71) -- cycle ;
\draw   (504.73,3152.93) .. controls (504.73,3152.4) and (505.1,3151.98) .. (505.54,3151.98) .. controls (505.99,3151.98) and (506.35,3152.41) .. (506.35,3152.94) .. controls (506.35,3153.47) and (505.98,3153.9) .. (505.54,3153.89) .. controls (505.09,3153.89) and (504.73,3153.46) .. (504.73,3152.93) -- cycle ;
\draw   (500.63,3176.84) .. controls (500.63,3176.31) and (501,3175.88) .. (501.45,3175.88) .. controls (501.89,3175.89) and (502.25,3176.32) .. (502.25,3176.85) .. controls (502.25,3177.37) and (501.89,3177.8) .. (501.44,3177.8) .. controls (500.99,3177.8) and (500.63,3177.37) .. (500.63,3176.84) -- cycle ;
\draw   (488.31,3163.93) .. controls (488.32,3163.4) and (488.68,3162.97) .. (489.13,3162.98) .. controls (489.57,3162.98) and (489.94,3163.41) .. (489.93,3163.94) .. controls (489.93,3164.47) and (489.57,3164.9) .. (489.12,3164.89) .. controls (488.67,3164.89) and (488.31,3164.46) .. (488.31,3163.93) -- cycle ;
\draw   (380.97,3169.97) .. controls (380.97,3169.96) and (381.11,3169.95) .. (381.28,3169.95) .. controls (381.45,3169.96) and (381.59,3169.97) .. (381.59,3169.98) .. controls (381.59,3169.99) and (381.45,3170) .. (381.28,3170) .. controls (381.11,3170) and (380.97,3169.99) .. (380.97,3169.97) -- cycle ;
\draw   (595.75,3169.03) .. controls (595.75,3168.5) and (596.11,3168.07) .. (596.56,3168.08) .. controls (597.01,3168.08) and (597.37,3168.51) .. (597.37,3169.04) .. controls (597.36,3169.57) and (597,3170) .. (596.55,3169.99) .. controls (596.11,3169.99) and (595.74,3169.56) .. (595.75,3169.03) -- cycle ;
\draw   (462.72,3185.06) .. controls (462.73,3184.53) and (463.09,3184.1) .. (463.54,3184.1) .. controls (463.98,3184.1) and (464.35,3184.53) .. (464.34,3185.06) .. controls (464.34,3185.59) and (463.98,3186.02) .. (463.53,3186.02) .. controls (463.08,3186.02) and (462.72,3185.59) .. (462.72,3185.06) -- cycle ;
\draw   (459.83,3142.94) .. controls (459.83,3142.41) and (460.2,3141.98) .. (460.65,3141.98) .. controls (461.09,3141.99) and (461.45,3142.42) .. (461.45,3142.95) .. controls (461.45,3143.47) and (461.09,3143.9) .. (460.64,3143.9) .. controls (460.19,3143.9) and (459.83,3143.47) .. (459.83,3142.94) -- cycle ;
\draw   (577.73,3195.7) .. controls (577.73,3195.17) and (578.1,3194.75) .. (578.54,3194.75) .. controls (578.99,3194.75) and (579.35,3195.18) .. (579.35,3195.71) .. controls (579.35,3196.24) and (578.98,3196.67) .. (578.54,3196.67) .. controls (578.09,3196.66) and (577.73,3196.23) .. (577.73,3195.7) -- cycle ;
\draw   (574.99,3137.34) .. controls (575,3136.82) and (575.36,3136.39) .. (575.81,3136.39) .. controls (576.25,3136.39) and (576.61,3136.82) .. (576.61,3137.35) .. controls (576.61,3137.88) and (576.25,3138.31) .. (575.8,3138.31) .. controls (575.35,3138.3) and (574.99,3137.87) .. (574.99,3137.34) -- cycle ;
\draw   (406.04,3137.88) .. controls (406.04,3137.35) and (406.41,3136.93) .. (406.85,3136.93) .. controls (407.3,3136.93) and (407.66,3137.36) .. (407.66,3137.89) .. controls (407.66,3138.42) and (407.29,3138.85) .. (406.85,3138.85) .. controls (406.4,3138.84) and (406.04,3138.41) .. (406.04,3137.88) -- cycle ;
\draw   (422.66,3133.36) .. controls (422.66,3132.83) and (423.02,3132.4) .. (423.47,3132.41) .. controls (423.91,3132.41) and (424.28,3132.84) .. (424.27,3133.37) .. controls (424.27,3133.9) and (423.91,3134.32) .. (423.46,3134.32) .. controls (423.01,3134.32) and (422.65,3133.89) .. (422.66,3133.36) -- cycle ;
\draw   (405.34,3201.62) .. controls (405.35,3201.09) and (405.71,3200.67) .. (406.16,3200.67) .. controls (406.6,3200.67) and (406.96,3201.1) .. (406.96,3201.63) .. controls (406.96,3202.16) and (406.6,3202.59) .. (406.15,3202.59) .. controls (405.7,3202.58) and (405.34,3202.15) .. (405.34,3201.62) -- cycle ;
\draw   (427.07,3202.89) .. controls (427.07,3202.36) and (427.43,3201.93) .. (427.88,3201.93) .. controls (428.33,3201.93) and (428.69,3202.36) .. (428.68,3202.89) .. controls (428.68,3203.42) and (428.32,3203.85) .. (427.87,3203.85) .. controls (427.42,3203.85) and (427.06,3203.41) .. (427.07,3202.89) -- cycle ;
\draw   (520.32,3142.27) .. controls (520.32,3141.75) and (520.69,3141.32) .. (521.13,3141.32) .. controls (521.58,3141.32) and (521.94,3141.75) .. (521.94,3142.28) .. controls (521.94,3142.81) and (521.57,3143.24) .. (521.13,3143.24) .. controls (520.68,3143.23) and (520.32,3142.8) .. (520.32,3142.27) -- cycle ;
\draw   (533.53,3135.97) .. controls (533.53,3135.44) and (533.89,3135.02) .. (534.34,3135.02) .. controls (534.79,3135.02) and (535.15,3135.45) .. (535.15,3135.98) .. controls (535.14,3136.51) and (534.78,3136.94) .. (534.33,3136.93) .. controls (533.89,3136.93) and (533.53,3136.5) .. (533.53,3135.97) -- cycle ;
\draw   (534.75,3201.23) .. controls (534.75,3200.7) and (535.11,3200.27) .. (535.56,3200.27) .. controls (536.01,3200.27) and (536.37,3200.71) .. (536.37,3201.23) .. controls (536.36,3201.76) and (536,3202.19) .. (535.55,3202.19) .. controls (535.11,3202.19) and (534.75,3201.76) .. (534.75,3201.23) -- cycle ;
\draw   (516.1,3190.96) .. controls (516.1,3190.43) and (516.47,3190.01) .. (516.92,3190.01) .. controls (517.36,3190.01) and (517.72,3190.44) .. (517.72,3190.97) .. controls (517.72,3191.5) and (517.36,3191.93) .. (516.91,3191.92) .. controls (516.46,3191.92) and (516.1,3191.49) .. (516.1,3190.96) -- cycle ;
\draw   (379.69,3170.05) .. controls (379.69,3169.38) and (380.27,3168.83) .. (380.99,3168.83) .. controls (381.7,3168.83) and (382.28,3169.38) .. (382.28,3170.05) .. controls (382.28,3170.73) and (381.7,3171.27) .. (380.99,3171.27) .. controls (380.27,3171.27) and (379.69,3170.73) .. (379.69,3170.05) -- cycle ;

\draw (480.23,3141.93) node [anchor=north west][inner sep=0.75pt]  [font=\small]  {$C$};
\draw (363,3160) node [anchor=north west][inner sep=0.75pt]  [font=\small,rotate=-359.71,xslant=0.05]  {$A$};
\draw (601,3161) node [anchor=north west][inner sep=0.75pt]  [font=\small]  {$B$};
\draw (379,3133) node [anchor=north west][inner sep=0.75pt]  [font=\footnotesize]  {$a_{1}$};
\draw (379.97,3194) node [anchor=north west][inner sep=0.75pt]  [font=\footnotesize]  {$b_{1}$};
\draw (497.27,3134) node [anchor=north west][inner sep=0.75pt]  [font=\footnotesize]  {$c_{1}$};
\draw (491,3183) node [anchor=north west][inner sep=0.75pt]  [font=\footnotesize]  {$d_{1}$};
\draw (396,3122) node [anchor=north west][inner sep=0.75pt]  [font=\footnotesize]  {$a_{2}$};
\draw (399.08,3206) node [anchor=north west][inner sep=0.75pt]  [font=\footnotesize]  {$b_{2}$};
\draw (510.57,3125) node [anchor=north west][inner sep=0.75pt]  [font=\footnotesize]  {$c_{2}$};
\draw (507,3196) node [anchor=north west][inner sep=0.75pt]  [font=\footnotesize]  {$d_{2}$};
\draw (415,3116) node [anchor=north west][inner sep=0.75pt]  [font=\footnotesize,rotate=-4.48,xslant=-0.12]  {$a_{3}$};
\draw (458,3125) node [anchor=north west][inner sep=0.75pt]  [font=\footnotesize,rotate=-4.48,xslant=-0.12]  {$a_{m}$};
\draw (421.91,3209) node [anchor=north west][inner sep=0.75pt]  [font=\footnotesize]  {$b_{3}$};
\draw (462.61,3191) node [anchor=north west][inner sep=0.75pt]  [font=\footnotesize]  {$b_{m}$};
\draw (525.08,3118) node [anchor=north west][inner sep=0.75pt]  [font=\footnotesize]  {$c_{3}$};
\draw (573.44,3121) node [anchor=north west][inner sep=0.75pt]  [font=\footnotesize]  {$c_{m}$};
\draw (525.28,3207) node [anchor=north west][inner sep=0.75pt]  [font=\footnotesize]  {$d_{3}$};
\draw (576.02,3201) node [anchor=north west][inner sep=0.75pt]  [font=\footnotesize]  {$d_{m}$};
\draw (431,3119) node [anchor=north west][inner sep=0.75pt]   [align=left] {.};
\draw (440,3120) node [anchor=north west][inner sep=0.75pt]   [align=left] {.};
\draw (448,3123) node [anchor=north west][inner sep=0.75pt]   [align=left] {.};
\draw (556,3210) node [anchor=north west][inner sep=0.75pt]   [align=left] {.};
\draw (455,3198) node [anchor=north west][inner sep=0.75pt]   [align=left] {.};
\draw (447,3205) node [anchor=north west][inner sep=0.75pt]   [align=left] {.};
\draw (437,3209) node [anchor=north west][inner sep=0.75pt]   [align=left] {.};
\draw (568,3207) node [anchor=north west][inner sep=0.75pt]   [align=left] {.};
\draw (544,3212) node [anchor=north west][inner sep=0.75pt]   [align=left] {.};
\draw (541,3119) node [anchor=north west][inner sep=0.75pt]   [align=left] {.};
\draw (552,3118) node [anchor=north west][inner sep=0.75pt]   [align=left] {.};
\draw (564,3121) node [anchor=north west][inner sep=0.75pt]   [align=left] {.};

\end{tikzpicture}
        \caption{Collapsed Faceted $3$-ball inserted $m$-vertices}
        \label{m.cell-PTFP}
    \end{subfigure}

    \caption{Cellulation of the collapsed faceted $3$-ball}
    \label{collapsed faceted $3$-ball}
\end{figure}

Due to modification on the faceted $3$-ball, the face-identification maps $\epsilon_m$ go under some modification in order to apply to the collapsed faceted $3$-ball. The modified face-identification maps are described as follows: 

\[
\epsilon_{1,m}^{*} :
\begin{pmatrix}
C & b_m & \cdots & b_1 & A & a_1 & a_2 & \cdots & a_m \\
C & d_1 & \cdots & d_m & B & c_m & c_{m-1} & \cdots & c_1
\end{pmatrix}
\]

\[
\epsilon_{2,m}^{*} :
\begin{pmatrix}
C & a_m & \cdots & a_1 & A & B & c_m & \cdots & c_2 & c_1 \\
b_m & b_{m-1} & \cdots & A & B & d_m & d_{m-1} & \cdots & d_1 & C
\end{pmatrix}.
\]

In the case $m=0$, the collapse of the edge $CD$ to a single vertex causes certain edges to become indistinguishable at the vertex level. To resolve this, we relabel the edges by $e_i$, following the notation shown in Figure~\ref{0.faceted_cell}. Under this relabeling, the associated face-identification maps are defined as follows:

\[
\epsilon_{1,0}^{*} : \begin{pmatrix}
e_2 & e_1  \\
e_4 & e_3
\end{pmatrix}
\qquad
\epsilon_{2,0}^{*} : \begin{pmatrix}
e_1 & e_5 & e_3\\
e_5 & e_4 & e_2
\end{pmatrix}.
\]

In the same manner, we introduce the maps $\epsilon_{1,m}^{*}$ and $\epsilon_{2,m}^{*}$ as the modified face-identification maps for the collapsed faceted $3$-ball, and the collection of all these maps is defined by $\epsilon_m^{*}=\{(\epsilon_{1,m}^{*})^{\pm1},(\epsilon_{2,m}^{*})^{\pm1}\}$.

Following the approach of Section~\ref{construct_manifolds}, we apply the modified face-identification maps $\epsilon_m^{*}$ to the collapsed faceted $3$-ball illustrated in Figure~\ref{m.cell-PTFP}. Then the resulting quotient space is a closed $3$-dimensional pseudomanifold with a single vertex, which we denote by $M(\epsilon_m^{*})$. Varying the parameter $m$ therefore yields a family of closed $3$-dimensional pseudomanifolds, each having exactly one vertex.

The following theorem, motivated by \cite{cannon2000introduction}, establishes that the resulting pseudomanifold $M(\epsilon_m^{*})$ has a unique vertex whose link is a torus rather than a sphere.  Moreover, it shows that removing this vertex from $M(\epsilon_m^{*})$ yields a non-compact $3$-manifold homeomorphic to a knot complement, which can be identified by computing its fundamental group.

\medskip

\begin{theorem}\label{Th:2.2}
Let $M_0(\epsilon_m^{*})$ be obtained by removing the vertex from the pseudomanifold $M(\epsilon_m^{*})$. Then $M_0(\epsilon_m^{*})$ is a  non-compact $3$-manifold  homeomorphic to the complement of a knot in $\mathbb{S}^3$. Moreover, the fundamental group $\pi_1\big(M_0(\epsilon_m^{*})\big)$ admits a presentation obtained from the presentation of $\pi_1\big(M(\epsilon_m)\big)$ by deleting the relators corresponding to the collapsed edge cycles.
\end{theorem}

\begin{proof}
It is obvious that applying the modified face-identification maps $\epsilon_m^{*}$ to the collapsed faceted $3$-ball produces a closed $3$-dimensional pseudomanifold. Now suppose that $f_i'$ denotes the number of $i$-cells in the resulting complex $M(\epsilon_m^{*})$. Collapsing an edge on the faceted $3$-ball reduces the number of edge cycles by one; hence, the new number of \( 1 \)-cells is \( f_1' =  1 \). Since collapsing edges does not increase the number of vertex classes, the new \( f_0 \) remains \( 1 \). It follows that the Euler characteristic of $M(\epsilon_m^{*})$ is equal to $1$. Therefore, $M(\epsilon_m^{*})$ is indeed a closed $3$-dimensional pseudomanifold with exactly one vertex.

Now removing that vertex yields a non-compact manifold, which is denoted by $M_0(\epsilon^{*}_m)$. The \( 1 \)-skeleton of \( M(\epsilon_m) \) consists of a single vertex and is isomorphic to a wedge of two circles. When passing from \( M(\epsilon_m) \) to \( M(\epsilon^{*}_m) \), one of these circles is collapsed to a point. Thus, the complement of the vertex in \( M(\epsilon^{*}_m) \) is homeomorphic to the complement of one circle in \( M(\epsilon_m) \). Since \( M(\epsilon_m) \) is homeomorphic to \( \mathbb{S}^3 \). Therefore, \( M_0(\epsilon^{*}_m) \) is homeomorphic to the complement of a knot in \( \mathbb{S}^3 \).

Finally, the assertion on the fundamental group follows from Lemma $\ref{Th:f.g}$, since the corresponding group relators vanish precisely when the associated edge cycles are collapsed.
\end{proof}

\begin{corollary}\label{f.g.vertex.complement}
Let $M_0(\epsilon_m^{*})$ be defined as above for each non-negative integer $m$.  
Then, by Theorem~$\ref{Th:2.2}$, the fundamental group of $M_0(\epsilon_m^{*})$ admits the following presentations.

\begin{itemize}
\item If $m$ is even, say $m = 2t$ for some $t \in \mathbb{N} \cup \{0\}$, then
\begin{equation}\label{f.g_vertex_complement1}
\pi_1(M_0(\epsilon_m^{*})) =
\left\langle x_1, x_2 \;\middle|\;
(x_2^{-1} x_1 x_2 x_1^{-1})^{t}
(x_2^{-1}x_1x_2x_1x_2^{-1})
(x_1^{-1}x_2x_1x_2^{-1})^{t}
\right\rangle.
\end{equation}

\item If $m$ is odd, say $m = (2t + 1)$ for some $t \in \mathbb{N} \cup \{0\}$, then
\begin{equation}\label{f.g_vertex_complement2}
\pi_1(M_0(\epsilon_m^{*})) =
\left\langle x_1, x_2 \;\middle|\;
(x_2 x_1^{-1} x_2^{-1} x_1)^{t+1}
(x_2)
(x_1 x_2^{-1} x_1^{-1} x_2)^{t+1}
\right\rangle.
\end{equation}
\end{itemize}
\end{corollary}

From Theorem \ref{Th:2.2}, if we remove this singular vertex from the pseudomanifold, the remaining space, called the \emph{vertex complement}, is homeomorphic to the complement of a \emph{knot} in $S^3$. Naturally, the following question arises:
\smallskip

\noindent\textit{For each non-negative integer $m$, which knot has a complement homeomorphic to the vertex complement $M_0(\epsilon_m^{*})$}?

\medskip

We now proceed to resolve this question. First, we state the main theorem, which provides a complete topological identification of the vertex complement arising from the above construction. The proof of this theorem is presented in Section~\ref{main_theorem_proof}. Before presenting the proof, we first assemble the necessary results that will be used throughout in the proof. The precise statement is given in the following theorem.

\begin{maintheorem}
For each non-negative integer $m$, let $M(\epsilon_m^{*})$ be the pseudomanifold defined above, and let $M_0(\epsilon_m^{*})$ denote its vertex complement obtained by deleting the singular vertex of $M(\epsilon_m^{*})$. Then, in the cases described below, the space $M_0(\epsilon_m^{*})$ is homeomorphic to the complement of a twist knot:
\begin{itemize}
    \item If $m = 0$, then $M_0(\epsilon_0^{*})$ is homeomorphic to the \textbf{trefoil knot} $(3_1)$ complement.

    \item If $m$ is even and $m\geq2$, 
    then $M_0(\epsilon_m^{*})$ is homeomorphic to the \textbf{$(m+3)_2$ knot} complement.

    \item If $m$ is odd and $m\geq1$, then $M_0(\epsilon_m^{*})$ is homeomorphic to the \textbf{$(m+3)_1$ knot} complement.
\end{itemize}
\end{maintheorem}

\medskip
\noindent\textbf{\emph{Note:}}
The knots $3_1$, $(m+3)_2$, and $(m+3)_1$, for the corresponding values of $m$, are labeled according to the conventions of the \href{https://katlas.org/wiki/The_Rolfsen_Knot_Table}{Rolfsen knot table}.

\section{Computation of the Fundamental Group of the Twist Knot Complement}\label{fundamental_group_knot}
In this section, we calculate the fundamental group of the twist knot complement. To do this computation, we first deform the twist knot in a suitable manner that simplifies the analysis of the fundamental group. We consider two cases, according to the choice of $n$: the case of an \(n_1\)-crossing twist knot, where \(n\) is even, and the case of an \(n_2\)-crossing twist knot, where \(n\) is odd. For notational convenience, we fix an orientation of the resulting diagram and assign labels to the arcs according to the convention illustrated in Figure \ref{General_knot}.

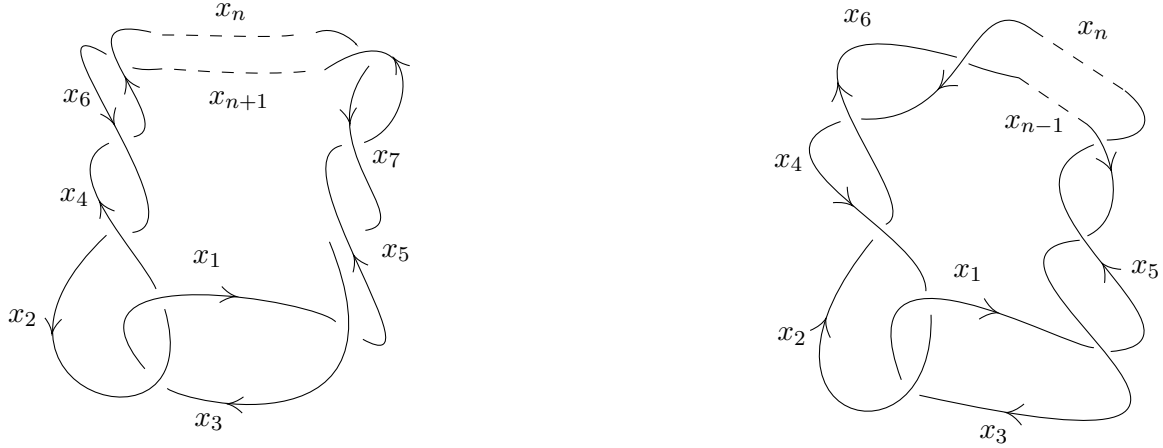
\begin{figure}[ht]
    \centering

    \begin{subfigure}[t]{0.45\textwidth}
        \centering
        \begin{tikzpicture}[scale=1pt, x=0.80pt, y=0.80pt, yscale=-1, xscale=1, every picture/.style={line width=0.75pt}, scale=0.75]
            \useasboundingbox (0,2470) rectangle (320,2800);
           
\draw    (73.39,2715.76) .. controls (12,2644) and (182.96,2673.42) .. (193.06,2686.34) ;
\draw    (210.15,2697.25) .. controls (261.44,2725.67) and (154.98,2584.28) .. (196.94,2575.67) ;
\draw    (210.93,2573.81) .. controls (247.45,2561.32) and (245.9,2488.12) .. (186.06,2528.31) ;
\draw    (65.61,2526.87) .. controls (71.83,2528.31) and (74.94,2528.31) .. (84.26,2527.59) ;
\draw    (65.61,2630.22) .. controls (109.91,2625.19) and (-8.21,2483.81) .. (49.3,2518.54) ;
\draw    (85.82,2678.3) .. controls (115.35,2778.78) and (-54.06,2729.26) .. (49.3,2631.94) ;
\draw    (80.38,2665.67) .. controls (71.05,2643.85) and (15.1,2590.03) .. (50.85,2574.24) ;
\draw    (66.39,2570.22) .. controls (97.48,2564.91) and (12.77,2489.55) .. (77.27,2504.62) ;
\draw    (180.62,2504.91) .. controls (195.39,2498.88) and (208.6,2514.67) .. (205.49,2511.8) ;
 \draw    (214.04,2525.72) .. controls (174.41,2575.1) and (245.9,2623.9) .. (211.71,2628.92) ;
\draw    (87.37,2728.11) .. controls (101.36,2738.01) and (242.79,2768.15) .. (189.17,2636.1) ;
\draw  [dash pattern={on 4.5pt off 4.5pt}]  (84.26,2506.34) .. controls (113.02,2506.2) and (160.42,2509.07) .. (171.3,2505.48) ;
\draw  [dash pattern={on 4.5pt off 4.5pt}]  (91.72,2527.87) .. controls (100.58,2528.45) and (144.1,2532.04) .. (174.41,2529.17) ;
\draw   (121.66,2664.13) .. controls (124.57,2667.35) and (127.62,2669.45) .. (130.83,2670.43) .. controls (127.47,2670.58) and (123.96,2671.85) .. (120.29,2674.24) ;
\draw   (203.96,2655.23) .. controls (205.11,2651.1) and (205.21,2647.47) .. (204.26,2644.35) .. controls (206.26,2646.97) and (209.3,2649.09) .. (213.41,2650.7) ;
\draw   (230.89,2531.49) .. controls (231.38,2527.24) and (230.89,2523.64) .. (229.47,2520.69) .. controls (231.85,2522.98) and (235.21,2524.63) .. (239.51,2525.62) ;
\draw   (134.43,2743.53) .. controls (131.36,2740.46) and (128.18,2738.52) .. (124.92,2737.7) .. controls (128.27,2737.39) and (131.72,2735.95) .. (135.25,2733.37) ;
\draw   (206.17,2551.96) .. controls (203.64,2555.47) and (202.27,2558.84) .. (202.06,2562.09) .. controls (201.1,2558.97) and (198.98,2555.98) .. (195.7,2553.11) ;
\draw   (21.16,2684.68) .. controls (17.96,2687.63) and (15.94,2690.69) .. (15.08,2693.83) .. controls (14.77,2690.59) and (13.3,2687.25) .. (10.66,2683.83) ;
\draw   (55.13,2551.83) .. controls (53.77,2555.9) and (53.5,2559.52) .. (54.29,2562.68) .. controls (52.42,2559.97) and (49.48,2557.71) .. (45.47,2555.92) ;
\draw   (60.61,2543.65) .. controls (61.84,2539.54) and (62,2535.92) .. (61.1,2532.78) .. controls (63.06,2535.43) and (66.07,2537.6) .. (70.13,2539.28) ;
\draw   (45.1,2621.67) .. controls (46.07,2617.5) and (46.01,2613.87) .. (44.93,2610.79) .. controls (47.04,2613.32) and (50.17,2615.31) .. (54.34,2616.76) ;

\draw (101.97,2639.83) node [anchor=north west][inner sep=0.75pt]    {$x_{1}$};
\draw (102.97,2742) node [anchor=north west][inner sep=0.75pt]    {$x_{3}$};
\draw (-14,2676.92) node [anchor=north west][inner sep=0.75pt]    {$x_{2}$};
\draw (18.42,2601.37) node [anchor=north west][inner sep=0.75pt]    {$x_{4}$};
\draw (220.34,2635.98) node [anchor=north west][inner sep=0.75pt]    {$x_{5}$};
\draw (20.41,2538.87) node [anchor=north west][inner sep=0.75pt]    {$x_{6}$};
\draw (214.37,2577.33) node [anchor=north west][inner sep=0.75pt]    {$x_{7}$};
\draw (115.9,2485) node [anchor=north west][inner sep=0.75pt]    {$x_{n}$};
\draw (111.88,2540) node [anchor=north west][inner sep=0.75pt]    {$x_{n+1}$};

        \end{tikzpicture}
        \vspace{4pt}
        \caption{$n_1$ Crossing Knot, where $n$ is even}
        \label{fig:n_1-crossing-knot}
    \end{subfigure}%
    \hfill
    \begin{subfigure}[t]{0.45\textwidth}
        \centering
        \begin{tikzpicture}[scale=1pt, x=0.80pt,y=0.80pt,yscale=-1,xscale=1, every picture/.style={line width=0.75pt}, scale=0.75]
            \useasboundingbox (300,2470) rectangle (620,2800);
            
\draw    (358.12,2513.73) ;
\draw    (417.31,2722.66) .. controls (379.2,2621.1) and (520.29,2702.51) .. (538.13,2702.51) ;
\draw    (548.67,2704.92) .. controls (625.7,2698.48) and (459.48,2605.78) .. (538.13,2575.15) ;
\draw    (546.24,2572.25) .. controls (570.56,2571.93) and (579.48,2559.03) .. (557.59,2541.3) ;
\draw    (392.17,2559.03) .. controls (460.29,2563.06) and (459.48,2472.79) .. (500.02,2503.42) ;
\draw    (432.72,2667.04) .. controls (434.34,2631.58) and (313.52,2583.21) .. (379.2,2560.65) ;
\draw    (374.33,2529.86) ;
\draw    (399.47,2634.8) .. controls (305.41,2749.26) and (439.2,2783.92) .. (435.96,2681.87) ;
\draw    (407.58,2625.13) .. controls (438.39,2611.43) and (294.87,2477.62) .. (452.18,2520.34) ;
\draw    (459.48,2524.37) .. controls (465.15,2526.79) and (478.13,2530.82) .. (491.1,2533.24) ;
\draw    (532.45,2561.45) .. controls (547.05,2570.32) and (563.27,2618.68) .. (534.07,2632.38) ;
\draw    (428.66,2732.33) .. controls (718.14,2796.81) and (427.85,2643.67) .. (529.21,2635.12) ;
\draw  [dash pattern={on 4.5pt off 4.5pt}]  (500.02,2503.42) -- (557.59,2541.3) ;
\draw  [dash pattern={on 4.5pt off 4.5pt}]  (491.1,2533.24) -- (532.45,2561.45) ;
\draw   (492.7,2751.12) .. controls (489.65,2747.02) and (486.43,2744.39) .. (483.03,2743.23) .. controls (486.6,2742.91) and (490.36,2741.11) .. (494.29,2737.83) ;
\draw   (373.66,2548.34) .. controls (376.36,2544) and (377.62,2540.05) .. (377.43,2536.48) .. controls (379.06,2539.66) and (382.13,2542.46) .. (386.66,2544.87) ;
\draw   (544.79,2661.88) .. controls (545.45,2656.82) and (544.95,2652.71) .. (543.3,2649.53) .. controls (546.11,2651.76) and (550.07,2653.05) .. (555.19,2653.38) ;
\draw   (452.89,2538.18) .. controls (447.78,2537.75) and (443.67,2538.43) .. (440.55,2540.21) .. controls (442.67,2537.32) and (443.78,2533.32) .. (443.9,2528.23) ;
\draw   (552.19,2579.6) .. controls (549.74,2584.09) and (548.72,2588.1) .. (549.11,2591.65) .. controls (547.29,2588.57) and (544.06,2585.96) .. (539.41,2583.81) ;
\draw   (469.33,2671.63) .. controls (471.86,2676.07) and (474.74,2679.06) .. (477.97,2680.61) .. controls (474.38,2680.51) and (470.44,2681.85) .. (466.14,2684.63) ;
\draw   (360.14,2691.71) .. controls (364.85,2689.68) and (368.14,2687.15) .. (370.04,2684.12) .. controls (369.55,2687.65) and (370.47,2691.69) .. (372.8,2696.24) ;
\draw   (383.4,2601.58) .. controls (383.53,2606.67) and (384.66,2610.66) .. (386.78,2613.54) .. controls (383.66,2611.77) and (379.55,2611.12) .. (374.44,2611.56) ;

\draw (448.48,2648.75) node [anchor=north west][inner sep=0.75pt]    {$x_{1}$};
\draw (464.69,2750) node [anchor=north west][inner sep=0.75pt]    {$x_{3}$};
\draw (338,2688.41) node [anchor=north west][inner sep=0.75pt]    {$x_{2}$};
\draw (336,2579.6) node [anchor=north west][inner sep=0.75pt]    {$x_{4}$};
\draw (560,2646.5) node [anchor=north west][inner sep=0.75pt]    {$x_{5}$};
\draw (379.55,2488.52) node [anchor=north west][inner sep=0.75pt]    {$x_{6}$};
\draw (480.49,2555) node [anchor=north west][inner sep=0.75pt]    {$x_{n-1}$};
\draw (526.32,2496.58) node [anchor=north west][inner sep=0.75pt]    {$x_{n}$};

        \end{tikzpicture}
        \vspace{4pt}
        \caption{$n_2$ Crossing Knot, where $n$ is odd}
        \label{fig:n_2-crossing-knot}
    \end{subfigure}

    \caption{Deformation of \(n_1\)- and \(n_2\)-crossing knots with labeling.}
    \label{General_knot}
\end{figure}

Now we aim to formulate the general form of these twist knot groups using the algorithm outlined above (see Algorithm \ref{Algorithm}). As observed in Lemma \ref{ref: wirtinger presentation}, the last relator is a consequence of the others and hence we can exclude it from computation. Thus, for an $n$-crossing knot, we require only $(n-1)$ relators. The procedure for finding relators is as follows:

\medskip
\noindent \textbf{Case $1$} (When $n$ is even):
Following the labeling of the arc as defined on Figure \ref{fig:n_1-crossing-knot}, we first consider an $n_1$-crossing knot complement, where $n$ is even. In this case, the first two relators are always
\begin{center}
$x_3=x_2^{-1}x_1x_2$\\
$x_4=x_1^{-1}x_2x_1$ 
\end{center}

\medskip
\noindent
If $n=4k$ for some $k\geq 1$, then the next $(n-4)$ relators are given by
\[
\left.
\begin{array}{l}
x_{4r+1} = x_{4r-1}^{-1}\; x_{4r-3}\; x_{4r-1}, \\[6pt]
x_{4r+2} = x_{4r}^{-1}\; x_{4r-2}\; x_{4r}, \\[6pt]
x_{4r+3} = x_{4r+1}\; x_{4r-1}\; x_{4r+1}^{-1}, \\[6pt]
x_{4r+4} = x_{4r+2}\; x_{4r}\; x_{4r+2}^{-1},
\end{array}
\right\}
\qquad \text{where } 1\leq r \leq \left\lfloor \frac{n-4}{4} \right\rfloor.
\]

The last relator is given by
$$x_{4k}= x_{4k-1}^{-1}\;x_{4k-3}\;x_{4k-1}.$$

\noindent
If $n=4k+2$ for some $k\geq 1$, then the next $(n-6)$ relators are given by
\[
\left.
\begin{array}{l}
x_{4r+1} = x_{4r-1}^{-1}\; x_{4r-3} \;x_{4r-1}, \\[6pt]
x_{4r+2} = x_{4r}^{-1} \;x_{4r-2}\; x_{4r}, \\[6pt]
x_{4r+3} = x_{4r+1}\; x_{4r-1}\; x_{4r+1}^{-1}, \\[6pt]
x_{4r+4} = x_{4r+2}\; x_{4r}\; x_{4r+2}^{-1},
\end{array}
\right\}
\qquad \text{where } 1\leq r \leq \left\lfloor \frac{n-6}{4} \right\rfloor.
\]

The last three relators are given by
\[
x_{4k+1} = x_{4k-1}^{-1}\; x_{4k-3} \;x_{4k-1}, \]
\[x_{4k+2} = x_{4k}^{-1}\; x_{4k-2}\; x_{4k}, \]
\[x_{4k+2}= x_{4k+1}\;x_{4k-1}\;x_{4k+1}^{-1}.\]




\medskip

Upon simplification, the group reduces to two generators and one relator of the following form:

\medskip
\noindent
If $n = 2(r+1)$ with $r$ odd, then

\begin{equation}
\pi_1(\mathbb{S}^{3} \setminus L) = \langle x_1, x_2 \mid (x_2x_1x_2^{-1}x_1^{-1})^r \, x_2x_1^{-1} \, (x_2^{-1}x_1x_2x_1^{-1})^r = 1 \rangle.
\label{eq:knot-group-odd}
\end{equation}

\medskip
\noindent
If $n = 2(r+1)$ with $r$ even, then

\begin{equation}
    \pi_1(\mathbb{S}^{3} \setminus L) = \langle x_1, x_2 \mid (x_2^{-1}x_1x_2x_1^{-1})^r \, x_2^{-1}x_1 \, (x_2^{-1}x_1^{-1}x_2x_1)^r = 1 \rangle.
\label{eq:knot-group-even}
\end{equation}
\bigskip

\noindent \textbf{Case $2$} (When $n$ is odd):
Now for finding the fundamental group of an $n_2$-crossing knot complement, where $n$ is odd and $n\geq3$, we consider the labeling of the arcs as defined on Figure \ref{fig:n_2-crossing-knot}. In this case, the first three relators are always
\begin{center}
$x_3=x_2x_1x_2^{-1}$\\
$x_4=x_1^{-1}x_2x_1$ \\
$x_5=x_3^{-1}x_1x_3$
\end{center}

\medskip
\noindent
If $n=4k+1$, where $k\geq 1$, then the next $(n-5)$ relators are given by
\[
\left.
\begin{array}{l}
x_{4r+2} = x_{4r}\; x_{4r-2}\; x_{4r}^{-1}, \\[6pt]
x_{4r+3} = x_{4r+1}\; x_{4r-1}\; x_{4r+1}^{-1}, \\[6pt]
x_{4r+4} = x_{4r+2}^{-1}\; x_{4r}\; x_{4r+2}, \\[6pt]
x_{4r+5} = x_{4r+3}^{-1}\; x_{4r+1}\; x_{4r+3},
\end{array}
\right\}
\qquad \text{where } 1\leq r \leq \left\lfloor \frac{n-5}{4} \right\rfloor.
\]

The last relator is given by
$$x_{4k+1}= x_{4k}\;x_{4k-2}\;x_{4k}^{-1}.$$

\noindent
If $n=4k+3$ for some $k\geq 1$, then the next $(n-7)$ relators are given by
\[
\left.
\begin{array}{l}
x_{4r+2} = x_{4r}\; x_{4r-2}\; x_{4r}^{-1}, \\[6pt]
x_{4r+3} = x_{4r+1}\; x_{4r-1}\; x_{4r+1}^{-1}, \\[6pt]
x_{4r+4} = x_{4r+2}^{-1}\; x_{4r}\; x_{4r+2}, \\[6pt]
x_{4r+5} = x_{4r+3}^{-1}\; x_{4r+1}\; x_{4r+3},
\end{array}
\right\}
\qquad \text{where } 1\leq r \leq \left\lfloor \frac{n-7}{4} \right\rfloor.
\]

The last three relators are given by
\[
x_{4k+2} = x_{4k}\; x_{4k-2} \;x_{4k}^{-1}, \]
\[x_{4k+3} = x_{4k+1}\; x_{4k-1}\; x_{4k+1}^{-1}, \]
\[x_{4k+3}= x_{4k+2}^{-1}\;x_{4k}\;x_{4k+2}.\]

\noindent Thus, if $L$ is an $n_2$-crossing knot, where $n$ is odd and $n = 2r + 3$ for some $r \in \mathbb{N} \cup \{0\}$, then the fundamental group of the complement of $L$ is of the following form:

When $r$ is odd, then
\begin{equation}\label{r_is_odd}
    \pi_1(\mathbb{S}^{3}\setminus L)= \langle x_1, x_2\;|\; (x_2x_1^{-1}x_2^{-1}x_1)^r\,(x_2x_1^{-1})\,(x_2^{-1}x_1^{-1}x_2x_1)^{r+1} \rangle
\end{equation}

When $r$ is even, then
\begin{equation}\label{r_is_even}
    \pi_1(\mathbb{S}^{3}\setminus L)= \langle x_1, x_2\;|\; (x_1^{-1}x_2^{-1}x_1x_2)^{r+1}\,(x_1x_2^{-1})\,(x_1^{-1}x_2x_1x_2^{-1})^{r} \rangle
\end{equation}

\section{Proof of The Main Theorem} \label{main_theorem_proof} 
With these preliminaries established, we are now in a position to describe precisely how the vertex complement arises as a twist knot complement. Now in order to make this precise, we now prove the main theorem:

\bigskip

\noindent {\em Proof of} {\bf Main Theorem}: 
We have already shown in Theorem \ref{Th:2.2} that $M_0(\epsilon_m^{*})$ is homeomorphic to a knot complement in $\mathbb{S}^3$. In this proof, we determine, for various values of $m$, which specific knot complement in $\mathbb{S}^3$ it corresponds to. We divide the proof into the following three cases corresponding to the value of $m$.

\medskip

\begin{itemize}
    \item [\ding{228}] \noindent
\textbf{Case 1} ($m=0$): The fundamental group of the trefoil knot complement $(3_1)$ is well known and can be computed as
\begin{equation}\label{f.g.trefoil_knot}
\pi_1\!\left(\mathbb{S}^{3} \setminus (3_1)\right)
= \left\langle x_1, x_2 \;\middle|\; x_2 x_1 x_2 x_1^{-1} x_2^{-1} x_1^{-1} \right\rangle.
\end{equation}

If we perform the substitutions $x_1^{-1} = x$ and $x_2 = y^{-1}x$ in \cref{f.g.trefoil_knot}, the relator simplifies to
$y^{-1} y^{-1} x y x = 1$. Rewriting further by substituting $y = x_2$ and $x = x_1$, we recover precisely the relation appearing in \cref{f.g_vertex_complement1} for $t=0$. Consequently, when $m=0$, the fundamental group of $M_0(\epsilon_0^{*})$ is isomorphic to the fundamental group of the trefoil knot complement.

\item[\ding{228}]\noindent
\textbf{Case 2} ($m$ is odd and $m \ge 1$):
By Corollary \ref{f.g.vertex.complement}, we derive a presentation for the fundamental group of the vertex complement $M_0(\epsilon_m^{*})$, where $m = 2t+1$ with $t \in \mathbb{N} \cup \{0\}$. Specifically, the fundamental group is given by
\begin{equation}\label{odd:vertex-complement}
    \pi_1\bigl(M_0(\epsilon^{*}_m)\bigr)
        = \left\langle x_1, x_2 \;\middle|\;
        (x_2x_1^{-1}x_2^{-1}x_1)^{t+1}
        \, x_2 \,
        (x_1x_2^{-1}x_1^{-1}x_2)^{t+1} 
        \right\rangle.
\end{equation}

From the presentations obtained earlier in \cref{eq:knot-group-odd} and \cref{eq:knot-group-even}, we can likewise describe the fundamental groups of knot complements of type $n_1$. Here the parameter $n$ is given by $n = 2(r+1)$, where $r \ge 1$.

Consequently, our analysis yields two distinct cases arising from the knot complements and a single case corresponding to the vertex complement. We proceed by considering each subcase individually.

\begin{description}
    \item[\underline{Subcase $2a$}] ($r$ is odd):
    Substitute $x_2x_1^{-1} = x$ and $x_1^{-1} = y$ into \cref{eq:knot-group-odd}.  
    This gives the relator
    \begin{equation}\label{eq:subcase-odd}
        (xy^{-1}x^{-1}y)^{r} \, x \, (yx^{-1}y^{-1}x)^{r}.
    \end{equation}
    Replacing $x = x_2$, $y = x_1$, and $r = (t+1)$ in \cref{eq:subcase-odd} yields precisely the same relation as in \cref{odd:vertex-complement}.

    \medskip
    \item[\underline{Subcase $2b$}] ($r$ is even):
    Substitute $x_2^{-1}x_1 = x$ and $x_2^{-1} = y$ into \cref{eq:knot-group-even}.  
    This gives the relator
    \begin{equation}\label{eq:subcase-even}
        (xy^{-1}x^{-1}y)^{r} \, x \, (yx^{-1}y^{-1}x)^{r}.
    \end{equation}
    Again, substituting $x = x_2$, $y = x_1$, and $r = (t+1)$ in \cref{eq:subcase-even} yields exactly the same relation as in \cref{odd:vertex-complement}.
\end{description}

\noindent
Thus, for odd $m$, the fundamental group of $M_0(\epsilon_m^{*})$ is isomorphic to the fundamental group of $(m+3)_1$ knot complement.

 \item [\ding{228}]\noindent
\textbf{Case 3} ($m$ is even and $m \ge 2$):  
By Corollary~\ref{f.g.vertex.complement}, for $m=2t$ with $t\in\mathbb{N}$, the fundamental group of the vertex complement $M_0(\epsilon_m^{*})$ admits the presentation
\begin{equation}\label{even.vertex_complement}
\pi_1(M_0(\epsilon_m^{*}))=
\left\langle x_1, x_2 \;\middle|\;
(x_2^{-1}x_1x_2x_1^{-1})^{t}
\,(x_2^{-1}x_1x_2x_1x_2^{-1})\,
(x_1^{-1}x_2x_1x_2^{-1})^{t}
\right\rangle .
\end{equation}
Replacing $x_1$ by $x$, and $x_2$ by $xy^{-1}$ in \cref{even.vertex_complement}, we obtain the relator as 
\begin{equation}\label{even.vc.update}
        (yxy^{-1}x^{-1})^{t}(yxy^{-1}xyx^{-1})(y^{-1}xyx^{-1})^{t}.
\end{equation}
As in the previous cases, the comparison with knot group presentations leads to two distinct subcases arising from knot complements, together with a single case associated to the vertex complement. We treat these cases separately.

\begin{description}
    \item[\underline{Subcase $3a$}] ($t$ is odd and $t \ge 1$):
The relator in~\cref{even.vc.update} can be simplified to the equivalent form
\begin{equation}\label{3a}
    (yxy^{-1}x^{-1})^{t}\,(yx)\,(y^{-1}xyx^{-1})^{t+1}.
\end{equation}
Performing the substitutions $x=x_1^{-1}$, $y=x_2$, and $t=r$ in~\cref{3a}, we recover precisely the defining relation appearing in~\cref{r_is_odd}.

\item[\underline{Subcase $3b$}] ($t$ even and $t \ge 2$):

Similarly, the relator in~\cref{even.vc.update} admits the equivalent expression
\begin{equation}\label{3b}
    (x^{-1}yxy^{-1})^{t+1}\,(xy)\,(x^{-1}y^{-1}xy)^{t}.
\end{equation}
Upon substituting $x=x_1$, $y=x_2^{-1}$, and $t=r$ in~\cref{3b}, we obtain exactly the same defining relation as in~\cref{r_is_even}.
 \end{description}
 This completes the argument and shows that, for even $m$, the fundamental group of $M_0(\epsilon_m^{*})$ is indeed isomorphic to that of the $(m+3)_2$ knot complement.
\end{itemize}

By Lemma \ref{not composite}, we know that if two knot complements have isomorphic fundamental groups, then both knots are either prime or composite. Since the fundamental groups of the knot complements in this family  are isomorphic to those of prime knot complements, it follows that all the knots in our case are prime. The result then follows from Lemma \ref{prime_knot}. 
\qed

\section{Ideal Triangulation of the Twist Knot Complement}\label{ideal_triangulation}
 
We now use the cellular structure together with the prescribed face identifications to construct a triangulation of the twist knot complement. In particular, the pseudomanifold structure of \(M(\epsilon_m^{*})\) naturally gives rise to an ideal triangulation of its vertex complement \(M_0(\epsilon_m^{*})\). The construction of this triangulation is formalized in the following theorem.

\begin{theorem}\label{theorem:ideal triangulation}
For each twist knot $K$, the complement $\mathbb{S}^3 \setminus K$ admits an ideal triangulation. Moreover, we develop an explicit algorithmic procedure to construct ideal triangulation of the complement $K$.
\end{theorem}

\begin{proof}
We begin with the cell structure of \(M(\epsilon_m^{*})\). To define a simplicial structure on \(M_0(\epsilon_m^{*})\), we first specify a simplicial structure on the faces of the faceted ball (Figure \ref{collapsed faceted $3$-ball}) in which the edge \(CD\) is collapsed to a single vertex \(C\). The resulting faceted \(3\)-ball has four faces. Before proceeding with the case-by-case analysis, we fix some notation that will be used throughout the proof. We denote a face by \([v_1 v_2 \cdots v_n]\), meaning that its boundary consists of the vertices \(v_1, v_2, \ldots, v_n\).

We first consider the case \(m \in \mathbb{N}\). In this case (Figure \ref{m.cell-PTFP}), the boundary faces $[Ca_ma_{m-1}\ldots a_1$ $Ab_1b_2\ldots b_m]$, and $[Cc_1c_2\ldots c_mBd_md_{m-1}\ldots d_1]$ are identified with each other. We consider \(C\) as the main vertex and assign it color \(1\). To define a simplicial structure, we then join \(C\) to all other vertices of both faces by adding edges to those vertices which are not adjacent to \(C\), and we alternately assign colors \(2\) and \(3\) to the remaining vertices. Thus, both faces are triangulated in a three-colorable way. To triangulate the other two faces $[Ca_ma_{m-1}\ldots a_1ABc_mc_{m-1}\ldots c_1]$, and $[Cb_mb_{m-1}\ldots b_1ABd_md_{m-1}\ldots d_1]$, we proceed as follows. For the first face, draw edges from $C$ to all the vertices of the face that are not adjacent to $C$, assign $1$ color to $C$, and assign colors $2$, and $3$ alternately to remaining vertices of the face. For the second face, since $C$ is identified with $b_m$ under the face-identification map $\epsilon_{2,m}^{*}$, we draw edges from $b_m$ to all the vertices of the face that are not adjacent to $b_m$, assign color $1$ to $b_m$, and alternately assign colors $2$, and $3$ colors to the remaining vertices. In this way, all faces of the faceted $3$-ball admit triangulations that are compatible with a proper three-coloring.

Now we consider the case for $m=0$. In this case (see Figure~\ref{0.faceted_cell}), the boundary face with edges \(e_1\) and \(e_2\) is identified with the boundary face having edges \(e_3\) and \(e_4\) via the face-identification map \(\epsilon_{1,m}^{*}\). To give simplicial structure on the faces, we introduce new vertices $v_1,v_2$ in the interiors of the respective faces. We then subdivide each face by joining \(v_1\) to the vertices \(C\) and \(A\), and similarly by joining \(v_2\) to the vertices \(C\) and \(B\). This subdivision yields a simplicial decomposition of the faces. Since \(C\) is the common vertex of these two faces, we assign color \(1\) to \(C\). The vertices \(A\) and \(B\) are assigned color \(2\), while the newly introduced vertices \(v_1\) and \(v_2\) are assigned color \(3\). Next, we consider the remaining boundary face whose edges are \(e_1, e_5, e_3\), which is identified, via the face-identification map $\epsilon_{2,m}^{*}$, with the boundary face having edges \(e_5, e_4, e_2\). These two faces are already $2$-simplices. We now assign the colours $1,2,3$ to the vertices $C,A,B$, respectively, on the face with edges $e_1, e_5, e_3$. For the face with edges $e_5, e_4, e_2$, we assign the colours $1,2,3$ to the vertices $A,C,B$, respectively. Thus the boundary faces receive a simplicial structure that is compatible with the given face identifications and admits a proper three-colouring.

To obtain a simplicial structure on the faceted $3$-ball, we place a vertex \( P \) in its interior of the faceted $3$-ball and join $P$ to each boundary vertex by an edge. If we assign the color \( 0 \) to the vertex \( P \), then the resulting triangulation becomes four-colorable. After performing the face identifications prescribed by $\epsilon_m^{*}$, this triangulated and collapsed faceted $3$-ball induces a simplicial structure on $M_0(\epsilon_m^{*})$. The resulting triangulation contains exactly two vertices. Finally, by contracting an edge incident to the interior vertex $P$, we obtain a one-vertex ideal triangulation of $S^3 \setminus K$.
\end{proof}

Using this triangulated model, we simplify the triangulation of each twist knot complement within \texttt{Regina}. This process yields a one-vertex ideal triangulation for every twist knot complement under consideration. Below, in Table \ref{table}, we present the resulting one-vertex ideal triangulations of $M_0(\epsilon_m^{*})$ for $0 \leq m \leq 14$, together with their corresponding isomorphism signature.

\begin{table}[ht]
\centering
\renewcommand{\arraystretch}{1.3}
\begin{tabular}{|
>{\centering\arraybackslash}p{2.6cm}|
>{\centering\arraybackslash}p{2cm}|
>{\centering\arraybackslash}p{2.2cm}|
>{\centering\arraybackslash}p{5cm}|}
\hline
\makecell{\textbf{Knot}\\\textbf{Complement}} &
\textbf{\(f\)-vector} &
\textbf{Volume} &
\makecell{\textbf{Isomorphic}\\\textbf{Signature}} \\
\hline
$3_1$  & (1, 2, 4, 2)   & possibly zero & {\footnotesize cPcbbbadu} \\
\hline
$4_1$  & (1, 2, 4, 2)   & 2.02988321    & {\footnotesize cPcbbbiht} \\
\hline
$5_2$  & (1, 3, 6, 3)   & 2.82812209    & {\footnotesize dLQbcccdero} \\
\hline
$6_1$  & (1, 4, 8, 4)   & 3.16396323    & {\footnotesize eLPkbcddddcwjb} \\
\hline
$7_2$  & (1, 4, 8, 4)   & 3.33174423    & {\footnotesize eLAkbccddmejln} \\
\hline
$8_1$  & (1, 5, 10, 5)  & 3.42720525    & {\footnotesize fLLQcacdeeenkmxsn} \\
\hline
$9_2$  & (1, 5, 10, 5)  & 3.48666015    & {\footnotesize fLAMcbccdeemejlhj} \\
\hline
$10_1$ & (1, 6, 12, 6)  & 3.52619599    & {\footnotesize gLLPQcdcfeffhsnlqws} \\
\hline
$11_2$ & (1, 6, 12, 6)  & 3.55381992    & {\footnotesize gLALQbcbefffxxjntxp} \\
\hline
$12_1$ & (1, 7, 14, 7)  & 3.57388254    & {\footnotesize hLALPkbcbegfgghhjslqws} \\
\hline
$13_2$ & (1, 7, 14, 7)  & 3.58891392    & {\footnotesize hLAMzkbccdefggmejlhhhj} \\
\hline
$14_1$ & (1, 8, 16, 8)  & 3.60046726    & {\footnotesize iLAwLQccbbdfhghhhhkhjeqsv}\\
\hline
$15_2$ & (1, 8, 16, 8) & 3.60953917 &
{\footnotesize iLAMLPcbcbdeghhhhhjhhwduh}\\
\hline
$16_1$ & (1, 9, 18, 9) & 3.61679305 & 
{\footnotesize jLAMLLQbcbdeghiiixxjxxjqenj}\\
\hline
$17_2$ & (1, 9, 18, 9) & 3.62268441 &
{\footnotesize jLAMzLQbcbdefhiiixxjxxxkulx}\\
\hline

\end{tabular}
\caption{ \(f\)-vectors, volumes, and isomorphic signatures associated with triangulations of knot complements.}\label{table}

\end{table}

\begin{remark}
The one-vertex ideal triangulation of the \((4_1)\)-knot complement obtained from our construction is isomorphic to the one-vertex ideal triangulation of the \((4_1)\)-knot complement described in \cite{dadd2016constructing}.
\end{remark}

From our algorithm and the software \texttt{Regina}, we obtain a one-vertex triangulation of the twist knot complement. In Matveev’s paper \cite{matveev2007algorithmic}, it is shown that any two one-vertex triangulations of the same manifold can be transformed into one another by a finite sequence of Pachner $2$-$3$ and $3$-$2$ moves. Consequently, we can construct infinitely many one-vertex triangulations of knot complements, which we state formally as follows.

\begin{corollary}
By applying Pachner the $2$-$3$ and $3$-$2$ moves, one can obtain infinitely many one-vertex triangulations of each twist knot complement.
\end{corollary}

\section{Minimality for Ideal Triangulation}\label{minimality}
From our construction in Theorem~\ref{theorem:ideal triangulation}, one obtains a one-vertex ideal triangulation for every twist knot complement. For such triangulations, we call a triangulation \emph{minimal} if it realises the smallest possible number of tetrahedra. Proving minimality is, in general, highly nontrivial, and determining minimal triangulations of \(3\)-manifolds remains an open and active area of research in low-dimensional topology.

However, every twist knot complement, except that of the trefoil knot, is a cusped hyperbolic \(3\)-manifold, since it admits a complete, finite-volume hyperbolic structure with a single cusp by Thurston's hyperbolization theorem \cite{thurston86hyperbolic,thurston1997three}. In computational topology, it is often valuable to have a well-organized collection of examples on which to perform experiments and test conjectures. An exhaustive catalogue of manifold triangulations of a given size and type referred to as a \emph{census}, serves as such a useful references. In dimension $3$, the \textit{SnapPea} census of cusped finite-volume hyperbolic \(3\)-manifolds constitutes a crucial dataset that has significantly advanced research in low-dimensional topology. The original one-vetrex triangulation census was created in 1989 by Hildebrand and Weeks for \(n \leq 5\) \cite{hildebrand89}, and was subsequently extended by Callahan, Hildebrand, and Weeks to the cases \(n = 6,7\), with a further extension to \(n = 8\) carried out by Thistlethwaite \cite{callahan99,thistlethwaite10}. Further, in \cite{burton14}, the author extends the complete census upto to $9$ tetrahedra and also presents the first list of all minimal triangulations of all census manifolds. The complete database, which includes all $229112$ minimal triangulations of the census manifolds upto $9$ tetrahedra, is available and can be downloaded from the webpage maintained at \href{https://people.smp.uq.edu.au/BenjaminBurton/code/#hypcensus}{\url{http://www.maths.uq.edu.au/\textasciitilde bab/code/}}. 

Using this census, we verified that all triangulations listed in Table~\ref{table}, except that of the trefoil knot, are indeed minimal. For completeness, we also provide below the relevant data for each manifold, as reported in the census database followed by \emph{the name of the manifold} (\textbf{MN}), \emph{its hyperbolic volume} (\textbf{V}), \emph{the length of its shortest geodesic} (\textbf{SG}), \emph{the Chern--Simons invariant} (\textbf{CS}), \emph{the first homology group} (\textbf{H}), and \emph{the full list of minimal triangulations (specified via their isomorphic signatures)} (\textbf{T, IS}):
\vspace{5pt}

\begin{table}[ht]
\centering
\renewcommand{\arraystretch}{1.3}
\begin{tabular}{|
>{\centering\arraybackslash}p{1.1cm}|
>{\centering\arraybackslash}p{2.3cm}|
>{\centering\arraybackslash}p{2.3cm}|
>{\centering\arraybackslash}p{2.3cm}|
>{\centering\arraybackslash}p{0.4cm}|
>{\centering\arraybackslash}p{0.4cm}|
>{\centering\arraybackslash}p{4.6cm}|}
\hline
\textbf{MN} & \textbf{V} & \textbf{SG} & \textbf{CS} &
\textbf{H} & \textbf{T} & \textbf{IS} \\
\hline
\texttt{m004} &
2.0298832128 &
1.0870701450 &
0.0000000000 &
$\mathbb{Z}$ &
1 &
\texttt{{\footnotesize cPcbbbiht}} \\
\hline
\texttt{m015} &
2.8281220883 &
0.5623991486 &
0.1532041333 &
$\mathbb{Z}$ &
1 &
\texttt{{\footnotesize dLQbcccdero}} \\
\hline
\texttt{m032} &
3.1639632289 &
0.3306355216 &
0.1559770167 &
$\mathbb{Z}$ &
2 &
\makecell{
\texttt{{\footnotesize eLAkbccddmejdj}}\\[2pt]
\texttt{{\footnotesize eLPkbcddddcwjb}}
} \\
\hline
\texttt{m053} &
3.3317442316 &
0.2171014650 &
0.0551535349 &
$\mathbb{Z}$ & 
1 &
\makecell{\texttt{{\footnotesize eLAkbccdeemejln}}} \\
\hline
\texttt{m074} &
3.4272052463 &
0.1535786928 &
0.2221331208 &
$\mathbb{Z}$ &
2 &
\makecell{
\texttt{{\footnotesize fLAMcbccdeemejdhn}}\\[2pt]
\texttt{{\footnotesize fLLQcacdeeenkmxsn}}
} \\
\hline
\texttt{m094} &
3.4866601463 &
0.1144571825 &
0.0077074477 &
$\mathbb{Z}$ &
1 &
\makecell{
\texttt{{\footnotesize fLAMcbccdeemejlhj}}
} \\
\hline
\texttt{s016} &
3.5261959907 &
0.0886385677 &
0.2422223193 &
$\mathbb{Z}$ &
2 &
\makecell{
\texttt{{\footnotesize gLALQbcbefffxxjntqt}}\\[2pt]
\texttt{\footnotesize gLLPQcdcfeffhsnlqws}
} \\
\hline
\texttt{s023} &
3.5538199185 &
0.0706938616 &
0.0200384940 &
$\mathbb{Z}$ &
1 &
\makecell{
\texttt{{\footnotesize gLALQbcbefffxxjntqt}}
} \\
\hline
\texttt{v0016} &
3.5738825406 &
0.0577105383 &
0.2200160608 &
$\mathbb{Z}$ &
2 &
\makecell{
\texttt{{\footnotesize hLALPkbcbegfgghhjslqws}}\\[2pt]
\texttt{{\footnotesize hLAMzkbccdefggmejdhhhn}}
} \\
\hline
\texttt{v0025} &
3.5889139178 &
0.0480106150 &
0.0382116994 &
$\mathbb{Z}$ &
1 &
\texttt{{\footnotesize  hLAMzkbccdefggmejlhhhj}} \\
\hline
\texttt{{\footnotesize t00017}} &
3.6004672628 &
0.0405711769 &
0.2048694742 &
$\mathbb{Z}$ &
2 &
\makecell{
\texttt{{\footnotesize iLAMLPcbcbdeghhhhhjhhwddq}}\\[2pt]
\texttt{{\footnotesize iLAwLQccbbdfhghhhhkhjeqsv}}
} \\
\hline
\texttt{{\footnotesize t00027}} &
3.6095391745 &
0.0347394042 &
0.0510292520 &
$\mathbb{Z}$ &
1 &
\texttt{{\footnotesize  iLAMLPcbcbdeghhhhhjhhwduh}} \\
\hline
\texttt{{\scriptsize o9\_00017}} &
3.6167930474 &
0.0300825642 &
0.1938823148 &
$\mathbb{Z}$ &
2 &
\makecell{
\texttt{{\footnotesize jLAMLLQbcbdeghiiixxjxxjqenj}}\\[2pt]
\texttt{{\footnotesize jLAMzLQbcbdefhiiixxjxxxkuuh}}
} \\
\hline

\texttt{{\scriptsize o9\_00022}} &
3.6226844082 &
0.0263045083 &
0.0605518627 &
$\mathbb{Z}$ &
1 &
\texttt{{\footnotesize  jLAMzLQbcbdefhiiixxjxxxkulx}} \\
\hline

\end{tabular}

\caption{Record of selected cusped hyperbolic \(3\)-manifolds.}
\label{table:census}
\end{table}

It is now straightforward to observe that identifying any two pairs of faces of a tetrahedron always yields a closed, connected manifold. Consequently, the number of tetrahedra required to triangulate the complement of the trefoil knot $3_1$ is at least two. Therefore, the given triangulations of the trefoil knot $3_1$, each consisting of two ideal tetrahedra, achieve this lower bound and are minimal.

From our construction, it is evident that the triangulations of the twist knot complements from $3_1$ up to $17_2$ are minimal. Given that our construction is systematic and proceeds by introducing a single vertex step by step along the collapsed $3$-ball, it is natural to conjecture that this approach produces minimal triangulations for all twist knots.  

\begin{conjecture} \label{minimal_conjecture}
Based on the construction described in Theorem $\ref{theorem:ideal triangulation}$, the one-vertex ideal triangulations of twist knot complements obtained via our algorithm are minimal, i.e., they use the smallest possible number of tetrahedra among all ideal triangulations of the same complement.
\end{conjecture}

\section{Conclusion and Future Exploration}\label{conclusion}
This work develops a unified topological framework, different from Thurston’s algorithm, for twist knot complements, combining geometric insight with analytic precision. The results provide a coherent topological and combinatorial description of these complements. We conclude with remarks and open questions intended to motivate further research in the study of $3$-manifolds and knot complements.

Moreover, a \emph{geometric triangulation} is a much stronger notion than an ideal triangulation. It is an ideal triangulation of a cusped hyperbolic $3$-manifold $M$ in which each tetrahedron is positively oriented and endowed with a hyperbolic structure of strictly positive volume, such that the gluing yields $M$ with its complete hyperbolic metric. It is further conjectured that every cusped hyperbolic $3$-manifold admits a triangulation that is convex with respect to the hyperbolic metric, i.e., a geometric triangulation. By employing \texttt{Regina}, we verify that the triangulations obtained in our construction are indeed geometric, as stated in Remark \ref{geometric_remark}. Moreover, since our algorithmic procedure for constructing twist knot complements is uniform across all cases, we conjecture that all ideal triangulations produced by our method are geometric, as formulated in Conjecture \ref{geometric_conjecture}.

\begin{remark}\label{geometric_remark}
    We are given explicit ideal triangulations of the knot complements corresponding to $3_1$, $4_1$, $5_2$, $6_1$, $7_2$, $8_1$, $9_2$, $10_1$, $11_2$, $12_1$, $13_2$, $14_1$, $15_2$, $16_1$ and $17_2$. Using the software \texttt{Regina}, we verify that each of these given ideal triangulations is \textit{geometric}, that is, all tetrahedra are positively oriented. Furthermore, by applying Pachner moves (specifically, $2$-$3$ and $3$-$2$ moves), we can obtain infinitely many distinct geometric ideal triangulations for each of these knot complements.
\end{remark}

\begin{conjecture} \label{geometric_conjecture}

Following the discussion in Remark $\ref{geometric_remark}$, since our procedure for obtaining the family of twist knots is both unique and consistent, we further conjecture that the ideal triangulations of the twist knot complements produced by our algorithm are geometric. In other words, for each twist knot generated through this construction, the associated ideal triangulation consists entirely of positively oriented tetrahedra.
\end{conjecture}

\begin{remark}
    Apart from these developments, the authors in \cite{cannon2000introduction} posed the following fundamental question concerning partial twisted face-pairing constructions:
\begin{quote}
\textbf{Question 1'.} \emph{Which knot complements are partial twisted face-pairing manifolds?}
\end{quote}
Since our model identification naturally arises from a construction of the partial face-pairing type, our results provide a partial answer to this question. 
In particular, we demonstrate that all twist knot complements can be realized as partial twisted face-pairing manifolds.
\end{remark}

The construction of triangulations for $3$-manifolds, together with the problem of identifying when such triangulations are minimal, has been a longstanding focus in combinatorial and geometric topology. In this work, we have developed an explicit and systematic construction of ideal triangulations for twist knot complements and shown that, in certain cases, these triangulations are facet-minimal. These results naturally prompt the following fundamental question:

\begin{question}
Can the ideal triangulations produced by our construction be shown to be minimal through a finite sequence of admissible transformations, such as Pachner moves?
\end{question}
\noindent We expect that a more detailed combinatorial and geometric investigation of these triangulations, potentially incorporating an analysis of bistellar moves and their influence on volume and shape parameters, could provide additional insight into this problem.

\begin{question}
Is it possible to identify a faceted $3$-ball and an associated system of face-pairing identifications that produce analogous topological constructions for other families of knot or link complements?
\end{question}

\begin{question}
As noted in~\cite{budney12}, triangulations of $2$-knots can be constructed. A natural question is whether the present construction can be extended to a faceted $4$-ball to produce a corresponding topological model for $2$-knot complements.
\end{question}

\noindent From a topological perspective, such an extension would provide a higher-dimensional analogue of the present $3$-manifold construction, potentially yielding new combinatorial models for $2$-knots in $\mathbb{S}^4$. Geometrically, it could offer insights into the relationships among piecewise-linear $4$-manifolds, handle decompositions, and smooth or locally flat embeddings of $2$-spheres in $\mathbb{S}^4$. Developing this framework would enhance our understanding of the interaction between combinatorial and geometric structures in higher-dimensional knot theory.

\bigskip

 \noindent {\bf Acknowledgement:} The second and third authors are supported by the institute fellowship at IIT Delhi, India. The authors thank the free version of ChatGPT for assistance in checking typographical errors and proofreading the manuscript.
 
\medskip

\noindent {\bf Data availability:} The authors declare that all data supporting the findings of this study are available within the article.

\medskip

\noindent {\bf Declarations}

\smallskip

\noindent {\bf Conflict of interest:} No potential conflict of interest was reported by the authors.


\begin{thebibliography}{999}

\bibitem{aribi23} F. B. Aribi, F. Guéritaud and E. P. Nakazawa, Geometric triangulations and the Teichm\"uller TQFT volume conjecture for twist knots, {\em Quantum Topol.} {\bf 14(2)} (2023), 285--406. \href{https://doi.org/10.4171/qt/178}{MR4636640}

\bibitem{bst1}  B. Basak,  M. Saha, and A. Trivedi, Topological and Combinatorial Constructions of \((p,2)\)-Torus Knot Complements, 2026, 19 pages (preprint).

\bibitem{BS2020}  B. Basak and E. Swartz, Three-dimensional normal pseudomanifolds with relatively few edges, {\em Adv. Math.} ~\textbf{365} (2020) 107035, 25pp.   \href{https://doi.org/10.1016/j.aim.2020.107035}{MR4064764}

\bibitem{bing59} R. H. Bing, An alternative proof that $3$-manifolds can be triangulated, {\em Ann. of Math.} {\bf 69} (1959), 37--65. \href{https://doi.org/10.2307/1970092}{MR0100841}

\bibitem{budney12} R. D. Budney, B. A. Burton and J. A. Hillman, Triangulating a Cappell-Shaneson knot complement, {\em Math. Res. Lett.} {\bf 19(5)} (2012), 1117--1126. \href{https://dx.doi.org/10.4310/MRL.2012.v19.n5.a12}{MR3039834}

\bibitem{burton14} B. A. Burton, The cusped hyperbolic census is complete, (2014). \href{https://doi.org/10.48550/arXiv.1405.2695}{arXiv.1405.2695}

\bibitem{burton2013regina}
B. A. Burton, R. Budney, W. Pettersson, \emph{et al.},
\emph{Regina: Software for 3-manifold topology and normal surface theory},
Available at \url{https://regina-normal.github.io/}, 1999–2013.

\bibitem{callahan99} P. J. Callahan, M. V. Hildebrand and J. R. Weeks, A census of cusped hyperbolic $3$-manifolds, {\em Math. Comp.}, {\bf 68(225)} (1999), 321--332. \href{https://www.ams.org/journals/mcom/1999-68-225/S0025-5718-99-01036-4/}{MR1620219}

\bibitem{cannon2000introduction}  J. W. Cannon, W. J. Floyd, and W. R. Parry, Introduction to twisted face-pairings, {\em Math. Res. Lett.},{\bf 7(4)} (2000), 477--492. \href{https://dx.doi.org/10.4310/MRL.2000.v7.n4.a14}{MR1783626}



\bibitem{dadd2016constructing} B. Dadd and A. Duan, Constructing infinitely many geometric triangulations of the figure eight knot complement, {\em Proc. Amer. Math. Soc.},{\bf 144(10)} (2016), 4545--4555. \href{https://doi.org/10.1090/proc/13076}{MR3531201}

\bibitem{FeustelWhitten1978} C. D. Feustel and W. C. Whitten Jr., Groups and complements of knots, {\em Canadian J. Math.}, {\bf 30(6)} (1978), 1284--1295. \href{https://doi.org/10.4153/CJM-1978-105-0}{MR0511562}



\bibitem{ibarra25} D. Ibarra, D. V. Mathews, and J. S. Purcell,
\emph{On geometric triangulations of double twist knots}, (2025). \href{https://arxiv.org/pdf/2104.03751}{arXiv:2504.09901}


\bibitem{mca1989knots} C. M. A. Gordon and J. Luecke, Knots are determined by their complements, {\em J. Amer. Math. Soc.}, {\bf 2(2)} (1989), 371--415. \href{https://doi.org/10.1090/S0894-0347-1989-0965210-7}{MR0965210}



\bibitem{hildebrand89} M. V. Hildebrand and J. R. Weeks, A computer generated census of cusped hyperbolic $3$-manifolds, {\em Computers and mathematics (Cambridge, MA, 1989)}, 53--59. \href{https://doi.org/10.1007/978-1-4613-9647-5_7}{MR1005959}

\bibitem{jaco20030} W. H. Jaco and J. H. Rubinstein, $0$-efficient triangulations of 3-manifolds, {\em  J. Differential Geom.}, {\bf 65(1)} (2003), 61--168. \href{https://doi.org/10.4310/jdg/1090503053}{MR2057531}

\bibitem{jaco2020minimal} W. H. Jaco, J. H. Rubinstein, J. Spreer, and S. Tillmann,  On minimal ideal triangulations of cusped hyperbolic 3-manifolds, {\em J. Topol.} {\bf 13(1)} (2020), 308--342. \href{https://doi.org/10.1112/topo.12127}{MR4138740}


\bibitem{matveev2007algorithmic} S. V. Matveev, {\em Algorithmic topology and classification of 3-manifolds}, second edition, Algorithms and Computation in Mathematics, 9, Springer, Berlin, 2007. \url{https://dl.acm.org/doi/10.5555/1209356}


\bibitem{moise52} E. E. Moise, Affine structures in $3$-manifolds. V. The triangulation theorem and Hauptvermutung, {\em Ann. of Math.}, {\bf 56} (1952), 96--114. \href{https://doi.org/10.2307/1969769}{MR0048805}

\bibitem{pachner1991pl} U. Pachner, P.L. homeomorphic manifolds are equivalent by elementary shellings, {\em  European J. Combin.}, {\bf 12(2)} (1991), 129--145. \href{https://doi.org/10.1016/S0195-6698(13)80080-7}{MR1095161}

\bibitem{piergallini1988standard} R. Piergallini, Standard moves for standard polyhedra and spines, {\em Rend. Circ. Mat. Palermo (2) Suppl.}, {\bf 18} (1988), 391--414. \href{https://hdl.handle.net/11581/244540}{MR0958750}


\bibitem{Poincare1} J.~W. Morgan and G. Tian, {\em Ricci flow and the Poincar\'e{} conjecture}, Clay Mathematics Monographs, 3, Amer. Math. Soc., Providence, RI, 2007 Clay Math. Inst., Cambridge, MA, 2007. \href{https://www.claymath.org/wp-content/uploads/2022/03/Ricci-pdf.pdf}{MR2334563}

\bibitem{Poincare2}  G. Perelman, Finite extinction time for the solutions to the Ricci flow on certain
three-manifolds, (2003). \href{https://doi.org/10.48550/arXiv.math/0307245}{arXiv:math/0307245}

\bibitem{rolfsen2003knots} D. Rolfsen, {\em Knots and Links}. American Mathematical Soc., 2003. \href{https://scispace.com/papers/knots-and-links-y6qvc0i2tv}{MR0515288}

\bibitem{schultens2014introduction} J. Schultens, {\em Introduction to 3-Manifolds}, volume 151. American Mathematical Soc., 2014. \href{https://doi.org/10.1090/GSM/151}{MR3203728}

\bibitem{thistlethwaite10} M. Thistlethwaite, \emph{Cusped hyperbolic manifolds with 8 tetrahedra}, available at \href{https://web.math.utk.edu/~mthistle/}{\url{https://web.math.utk.edu/~mthistle/}}, October 2010.


\bibitem{thistlethwaite25} M. B. Thistlethwaite, The enumeration and classification of prime 20-crossing knots, {\em  Algebr. Geom. Topol.} {\bf 25(1)} (2025), 329--344. \href{https://msp.org/agt/2025/25-1/p12.xhtml}{MR4877260}



\bibitem{thurston86hyperbolic} W. P. Thurston, Hyperbolic Structures on $3$-manifolds, I: Deformation of acylindrical manifolds, {\em Annals of Mathematics}, {\bf 124} (1986), 203--246. \href{https://doi.org/10.48550/arXiv.math/9801019}{MR0855294}

\bibitem{thurston1997three} W. P. Thurston, {\em Three-Dimensional Geometry and Topology, Vol. 1}, Princeton University Press, Princeton, 1997. \href{https://www.math.unl.edu/~jkettinger2/thurston.pdf}{MR1435975}

\end{thebibliography}
\end{document}